\documentclass{article}
\usepackage{amsmath,amsfonts,latexsym, amsrefs,amssymb}

\usepackage{graphicx}

\DeclareMathOperator{\Prob}{\mathbb{P}}   
\DeclareMathOperator{\erfc}{erfc}
\DeclareMathOperator{\Vol}{Vol}
\DeclareMathOperator{\cum}{cum}
\DeclareMathOperator{\A}{d A}

\newcommand{\1}{\mathds{1}}
\usepackage{enumerate}

\usepackage{wrapfig}

\usepackage{color}

\numberwithin{equation}{section}

\newcommand{\rd}{{\rm d}}
\newcommand{\rdA}{{\rm d A}}

\newcommand{\by}{{\bf{y}}}

\newcommand{\bv}{{\bf{v}}}

\newcommand{\al}{\alpha}

\newcommand{\be}{\begin{equation}}
\newcommand{\ee}{\end{equation}}

\newcommand{\hp}[1]{with $ #1 $-high probability}

\newcommand{\e}{{\varepsilon}}

\newcommand{\T}{\mathbb T}

\newcommand{\bT}{\T}
\newcommand{\non}{\nonumber}

\usepackage{amsmath} 
\usepackage{amssymb}
\usepackage{amsthm}
\usepackage{amsxtra}
\usepackage{amscd}
\usepackage{bbm}
\usepackage{mathrsfs}
\usepackage{bm}

\renewcommand{\b}[1]{\bm{\mathrm{#1}}} 
\newcommand{\bb}{\mathbb} 

\renewcommand{\cal}{\mathcal}

\newcommand{\mG}{\mathcal G}

\newcommand{\ii}{\mathrm{i}} 
\newcommand{\dd}{\mathrm{d}}

\newcommand{\col}{\mathrel{\mathop:}}
\newcommand{\st}{\,\col\,}
\newcommand{\deq}{\mathrel{\mathop:}=}

\renewcommand{\epsilon}{\varepsilon}
\renewcommand{\leq}{\leqslant}
\renewcommand{\geq}{\geqslant}

\renewcommand{\le}{\leq}
\renewcommand{\ge}{\geq}

\renewcommand{\P}{\mathbb{P}}
\newcommand{\E}{\mathbb{E}}
\newcommand{\R}{\mathbb{R}}
\newcommand{\C}{\mathbb{C}}
\newcommand{\N}{\mathbb{N}}

\newcommand{\qB}[1]{\Bigl[{#1}\Bigr]}

\newcommand{\hb}[1]{\bigl\{{#1}\bigr\}}

\newcommand{\hbb}[1]{\biggl\{{#1}\biggr\}}

\newcommand{\absb}[1]{\bigl\lvert #1 \bigr\rvert}

\DeclareMathOperator{\tr}{Tr}

\DeclareMathOperator{\supp}{supp}

\DeclareMathOperator{\re}{Re}
\DeclareMathOperator{\im}{Im}

\DeclareMathOperator{\OO}{O}
\DeclareMathOperator{\oo}{o}

\theoremstyle{plain} 
\newtheorem{theorem}{Theorem}[section]
\newtheorem*{theorem*}{Theorem}
\newtheorem{lemma}[theorem]{Lemma}
\newtheorem*{lemma*}{Lemma}

\newtheorem*{corollary*}{Corollary}
\newtheorem{proposition}[theorem]{Proposition}
\newtheorem*{proposition*}{Proposition}

\newtheorem{definition}[theorem]{Definition}
\newtheorem*{definition*}{Definition}

\newtheorem*{example*}{Example}
\newtheorem{remark}[theorem]{Remark}

\newtheorem*{remark*}{Remark}
\newtheorem*{remarks*}{Remarks}

\makeatletter
\renewcommand{\section}{\@startsection
{section}
{1}
{0mm}
{-2\baselineskip}
{1\baselineskip}
{\normalfont\large\scshape\centering}} 
\makeatother

\makeatletter
\renewcommand{\subsection}{\@startsection
{subsection}
{2}
{0mm}
{-\baselineskip}
{0\baselineskip}
{\normalfont\bf} } 
\makeatother

\usepackage[T1]{fontenc}

\usepackage{dsfont}
\usepackage{stmaryrd}

\newcommand{\nc}{\normalcolor}

\begin{document}

\title{{\sc\Large The local circular law II: the edge case}\vspace{1cm}}

\date{}

\author{\vspace{0.5cm}\normalsize{\sc Paul Bourgade}${}^1$\thanks{Partially supported by NSF grant DMS-1208859}\quad\quad
{\sc Horng-Tzer Yau}${}^1$\thanks{Partially supported
by NSF grants DMS-0757425, 0804279}\quad\quad
{\sc Jun Yin}${}^2$\thanks{Partially supported
by NSF grant DMS-1001655}
 \\\\
\normalsize Department of Mathematics, Harvard University\\
\normalsize Cambridge MA 02138, USA \\ \normalsize  bourgade@math.harvard.edu \quad
\normalsize htyau@math.harvard.edu ${}^1$ \quad  \\ \\
\normalsize Department of Mathematics, University of Wisconsin-Madison \\
\normalsize Madison, WI 53706-1388, USA \ \normalsize jyin@math.wisc.edu ${}^2$\vspace{1cm}}

\maketitle

\begin{abstract}
In the first part of this article \cite{BouYauYin2012Bulk}, we proved a local version of the circular law up to the finest scale $N^{-1/2+ \e}$
for  non-Hermitian random matrices at any point $z \in \C$ with $||z| - 1| > c $ for any $c>0$ independent of
the size of the matrix. Under the main assumption that the first three moments of the matrix elements
match those of a standard Gaussian random variable after proper rescaling, we extend this result to include the edge case  $ |z|-1=\oo(1)$.
Without the vanishing third moment assumption, we prove that the circular law is valid near the spectral edge $ |z|-1=\oo(1)$
up to scale $N^{-1/4+ \e}$.
\end{abstract}

\vspace{1.5cm}


\medskip

\medskip


\medskip

\newpage

 \section{Introduction}

The circular law in random matrix theory describes the macroscopic limiting spectral measure of normalized non-Hermitian matrices with
independent entries. Its origin goes beck to the work of Ginibre \cite{Gin1965},
who found the joint density of the eigenvalues of such Gaussian matrices. More precisely, for an
$N\times N$ matrix with independent  entries $\frac{1}{\sqrt {N}}z_{ij}$ such that  $z_{ij}$ is
identically distributed according to the measure $\mu_g=\frac{1}{\pi}e^{-|z|^2}\rdA(z)$
($\rdA$ denotes the Lebesgue measure on $\mathbb{C}$),
its eigenvalues $\mu_1,\dots,\mu_N$ have a probability density
proportional to
\begin{equation}\label{eqn:Ginibre}
\prod_{i<j}|\mu_i-\mu_j|^2e^{-N\sum_{k}|\mu_k|^2}
\end{equation}
with respect to the Lebesgue measure on $\C^{N}$.
These random spectral measures define a determinantal point process with the
explicit kernel (see \cite{Gin1965}) 
\begin{equation}\label{eqn:GinibreKernel}
K_N(z_1,z_2)=\frac{N}{\pi}e^{-\frac{N}{2}(|z_1|^2+|z_2|^2)}\sum_{\ell=0}^{N-1}\frac{(N z_1\overline{z_2})^\ell}{\ell!}
\end{equation}
with respect to the Lebesgue measure on $\C$. This integrability property allowed Ginibre to derive the  circular law for the eigenvalues,
i.e.,
$\frac{1}{N}\rho_1^{(N)}$ converges to the uniform measure on the unit circle,
\begin{equation}\label{eqn:circular}
\frac{1}{\pi}\1_{|z|<1}\rdA(z).
\end{equation}
This limiting law also holds for real Gaussian entries \cite{Ede1997},
for which a more detailed analysis was performed in \cite{ForNag2007,Sin2007,BorSin2009}.

For non-Gaussian entries,
Girko \cite{Gir1984} argued that the macroscopic limiting spectrum is still given by
(\ref{eqn:circular}).  His main insight is commonly known as  the {\it Hermitization technique}, which
converts the convergence of complex empirical measures into the convergence of logarithmic transforms of
a family of Hermitian matrices.  If we denote the original non-Hermitian matrix by $X$ and the eigenvalues of
$X$ by $\mu_j$,  then for any ${C}^2$ function $F$ we have the identity
\be\label{id0}
\frac 1 N \sum_{j=1}^N F (\mu_j) = \frac1{4\pi N} \int \Delta F(z)   \tr   \log (X^* - z^* ) (X-z)    \rdA(z).
\ee
Due to the logarithmic singularity at $0$, it is clear that the small eigenvalues of the Hermitian matrix $(X^* - z^* ) (X-z)  $
play a special role. A key question is to
estimate  the small eigenvalues of $(X^* - z^* ) (X-z)$, or in other words, the small singular values
of $ (X-z)$.
 This problem
was not treated  in \cite{Gir1984},
but the  gap was remedied in a series of  papers.
First Bai \cite{Bai1997} was able to treat the  logarithmic  singularity assuming  bounded density and bounded high moments
for the entries of the matrix (see also \cite{BaiSil2006}).
Lower bounds on the smallest singular values were given in Rudelson, Vershynin \cite{Rud2008,RudVer2008}, and subsequently
Tao, Vu \cite{TaoVu2008}, Pan, Zhou \cite{PanZho2010} and G\"otze, Tikhomirov \cite{GotTik2010} weakened the
moments and smoothness assumptions for the circular law, till the optimal $\mbox{L}^2$ assumption,
under which the circular law was proved in \cite{TaoVuKri2010}.

In the previous article \cite{BouYauYin2012Bulk}, we proved a local version of the circular law, up to the optimal scale
$N^{-1/2 + \e}$, in the bulk of the spectrum.
More precisely,
we considered   an $N \times N$  matrix $X$ with independent real\footnote{For the sake of notational simplicity we do not consider complex entries in this paper, but the statements and proofs are similar.} centered entries with variance $ N^{-1}$.
Let $\mu_j$, $j\in \llbracket 1, N\rrbracket$ denote the eigenvalues of $X$. To state the local circular law, we first define the
notion of \emph{stochastic domination}.

\begin{definition}
Let $W=(W_N)_{N\geq 1}$ be family a  random variables and $\Psi=(\Psi_N)_{N\geq 1}$  be  deterministic  parameters.
We say that $W$ is  \emph{stochastically dominated} by  $\Psi$
if for any   $  \sigma> 0$ and $D > 0$ we have
\begin{equation*}
 \P \qB{\absb{W_N} > N^\sigma \Psi_N  } \;\leq\; N^{-D}
\end{equation*}\nc
for sufficiently large $N$.
We denote this stochastic domination property by
\begin{equation*}
W \;\prec\; \Psi\,,\quad or \quad W =\OO_\prec (\Psi).
\end{equation*}
\end{definition}

In this paper, as in \cite{BouYauYin2012Bulk}, we assume that  the probability distributions of the   matrix elements
satisfy the following uniform subexponential decay property:
\be\label{subexp}
\sup_{(i,j)\in\llbracket 1,N\rrbracket^2}\Prob\left(|\sqrt{N}X_{i,j}|>\lambda\right)\leq \vartheta^{-1} e^{-\lambda^\vartheta }
\ee
for some constant $\vartheta >0$ independent of $N$.
This condition can of course be weakened to an hypothesis of boundedness on sufficiently high moments, but the error estimates
in the following Theorem would be weakened as well.

Let $f:\mathbb{C}\to\mathbb{R}$ be a fixed smooth compactly supported function,
and $f_{z_0}(\mu)=N^{2a}f(N^a(\mu-z_0))$, where    $z_0$ depends on $N$  and
$||z_0|-1|>\tau$ for some $\tau>0$ independent of $N$, and $a$ is a fixed scaling parameter in $(0,1/2]$.
Theorem 2.2 of  \cite{BouYauYin2012Bulk} asserts that  the following estimate holds:
\begin{equation}\label{yjgq}
\left( N^{-1} \sum_{j}f_{z_0} (\mu_j)-\frac1\pi \int f_{z_0}(z)    \, \rdA(z)  \right)\prec \ N^{-1+2a }.
\end{equation}
This implies that the circular law holds after zooming, in the bulk, up to scale $N^{-1/2+\e}$.
In particular, there are neither clusters of eigenvalues nor holes in the spectrum at such scales.

We aim at understanding the circular law close to the edge of the spectrum, i.e.,  $ |z_0|-1=\oo(1)$. The following is our main result.

\begin{theorem}\label{z1}
Let  $X$ be an $N\times N$ matrix with independent centered
entries of  variances $1/N$ and vanishing third moments.   Suppose that
the distributions of the matrix elements   satisfy the subexponential decay property (\ref{subexp}).
Let $f_{z_0}$ be defined as previously and $D$ denote the unit disk.
Then  for any $a \in (0,1/2]$ and any $z_0 \in \C$, we have
\be\label{yjgq2}
 \left( N^{-1} \sum_{j}f_{z_0} (\mu_j)-\frac1\pi \int_D f_{z_0}(z)\rdA(z)     \right)\prec \ N^{-1+2a }.
\ee
Notice that the main assertion of \eqref{yjgq2}  is for  $|z_0|-1=\oo(1)$ since the other cases were proved in  \cite{BouYauYin2012Bulk},
stated in \eqref{yjgq}.
\end{theorem}

Without the third moment vanishing assumption, we have the following weaker estimate. This estimate does not imply
the local law up to the finest scale $N^{-1/2+ \e}$, but it asserts that the circular law near the spectral edge
holds at least up to scale  $N^{-1/4+ \e}$.

\begin{theorem}\label{z1T}
Suppose that $X$ is an $N\times N$ matrix with independent centered
entries, variance $1/N$, satisfying the
subexponential decay property (\ref{subexp}).
Let $f_{z_0}$ be defined as previously, with $|z_0|\le C$, and $D$ denote the unit disk.
Then for any $a \in (0,1/4]$ and any $z_0 \in \C$, we have
\be\label{yjgq2T}
 \left( N^{-1} \sum_{j}f_{z_0} (\mu_j)-\frac1\pi \int_D f_{z_0}(z)\rdA(z)     \right)\prec \ N^{-1/2+2a }.
\ee
\end{theorem}

Shortly after the preprint \cite{BouYauYin2012Bulk} appeared, a version of
local circular law was proved by  Tao and Vu \cite{TaoVu2012} under the assumption
that  the first three moments matching a Gaussian distribution both in the bulk and near the edge.
Both  the assumptions and conclusions of Theorem 20 \cite{TaoVu2012}, when restricted to near the edge,
are thus very similar to  Theorem \ref{z1}. On the other hand, in the bulk case the assumption of vanishing third moments were not needed in \cite{BouYauYin2012Bulk}.
Our proof in this paper follows the approach of the companion article \cite{BouYauYin2012Bulk}, except
that for small eigenvalues
we will use the Green function comparison theorem \cite{ErdYauYin2010PTRF}. On the other hand, the method in  \cite{TaoVu2012}
relies on Jensen's formula  for determinants.

A main tool in the proof of the local circular law in \cite{BouYauYin2012Bulk} was a detailed analysis
of the self-consistent equations of the  Green functions
$$
G_{ij}(w) = [(X^* - z^* ) (X-z) - w]^{-1}_{ij}.
$$
We were able
to control $G_{ij}(E + \ii \eta)$ for the energy parameter $E$ in any compact set and for sufficiently  small $\eta$
so as to use the formula \eqref{id0} for functions $F$ at  scales $N^{-1/2+ \e}$.
We proved that, in particular, the Stieltjes transform $m=N^{-1}\sum G_{ii}$ converges to $m_{\rm c}(w,z)$, a fixed point of the self-consistent equation
\be\label{defmc1Y}
g(m)=\frac{1+w\, m(1+m)^2}{|z|^2-1}
\ee
when $||z| - 1|\ge \tau > 0$. In this region of $z$,
 the fixed point equation is stable.
However, for $ |z |-1=\oo(1)$,
  (\ref{defmc1Y}) becomes unstable when $|w|$ is small.  In this paper,  we will show that $m$ is still close to the  fix point of this  equation
under the additional  condition that the third moments of the matrix entries vanish. With this condition,
we can compare $m$ of our model with the corresponding quantity for the Gaussian ensemble. In order to carry out this comparison,
we prove   a local circular law for  the Ginibre ensemble near  the edge $ |z |-1=\oo(1)$ in Section 4. We now outline the  proof
of the main result, Theorem \ref{z1}.  Step 1: we use the  Hermitization technique to convert the Theorem \ref{z1} into the problem on the distribution of the eigenvalues $\lambda_i$'s  of $(X-z)^*(X-z)$. Step 2: for the eigenvalues of order 1, i.e. $\lambda_i\geq \e$ for some $\e>0$,  we will control  $\sum \log (\lambda_i)$   via an analysis
of the self-consistent equations of the  Green functions
$$
G_{ij}(w) = [(X^* - z^* ) (X-z) - w]^{-1}_{ij}.
$$
Step 3: for the eigenvalues of order $\oo(1)$, we will show that $\sum \log (\lambda_i)$ of our model has the same asymptotic distribution as the Ginibre ensemble, via the Green function comparison method first used in \cite{ErdYauYin2010PTRF}. The local circular law for the Ginibre ensemble is proved in Appendix \ref{App:Ginibre}. Step 4: combining
the previous steps, we conclude the proof of the local circular law on the edge.

\section{Proof of Theorem \ref{z1}  and  Theorem \ref{z1T}}

\subsection{Hermitization. }\label{sec:Hermitzation}
In the following, we will use the notation
$$
 Y_z=X-z I  \quad
$$
 where $I$ is the identity operator.
Let  $\lambda_j(z)$ be   the $j$-th eigenvalue (in the  increasing ordering) of $Y^*_z  Y_z $.
We will generally omit the $z-$dependence in these notations.
Thanks to the Hermitization technique of Girko \cite{Gir1984},
the first step in proving the local circular law is to understand the local statistics of eigenvalues of $Y^*_z Y_z$.
More precisely,  (\ref{id0}) yields
\be\label{id2}
N^{-1} \sum_j f_{z_0}(\mu_j)
= \frac1{4\pi} N^{-1+2a}\int (\Delta f)(\xi)  \sum_j \log \lambda_j(z)   \rdA(\xi),
\ee
where 
\begin{equation}\label{eqn:zXi}
z=z_0+N^{-a}\xi.
\end{equation}
Roughly speaking, for any fixed $\e>0$, the sum $N^{-1}\sum_{\lambda_i\ge \e}\log \lambda_i(z)$ will approach $\int_{x\ge \e} (\log x)\rho_{\rm c}(x)\rd x$,
where $\rho_c$ is the limiting spectral density of $Y_z^*Y_z$.   To estimate the error term in this convergence, we need to know some properties of $\rho_{\rm c}$ (Subsection \ref{subsec:equilibrium}) and, most importantly, the strong local Green function estimates of Subsection \ref{subsec:Green}.  The $\lambda_i\le \e$ part will be studied with the Green function comparison method in Section 3.

\subsection{Properties of $m_{\rm c}$ and $\rho_{\rm c}$. }\label{subsec:equilibrium}
 Define the Green function
 of $Y^*_z Y_z$  and its trace by
$$
G(w):=G(w,z)=(Y^*_z Y_z-w)^{-1},\quad  m(w):=m(w,z)=\frac{1}{N}\tr G(w,z)=\frac{1}{N}\sum_{j=1}^N\frac{1}{ \lambda_j(z) - w}, \quad w = E + \ii \eta.
$$
We will also need the following version of the Green function later on:
$$
\mG(w):=\mG(w,z)= ( Y_z Y^*_z-w)^{-1}.  \quad
$$
As we will see,  for some regions of $(w, z)$, with high  probability
$m(w,z)$ converges to  $m_{\rm c}(w,z)$ pointwise,
as $N\to \infty$ where $ m_{\rm c}(w,z)$  is  the unique
solution of
 \be\label{defmc1}
 m_{\rm c}^{-1}=-w(1+m_{\rm c})+|z|^2(1+m_{\rm c})^{-1}
\ee
with positive  imaginary  part (see Section 3 in \cite{GotTik2010} for the existence and uniqueness of such a solution).
The limit  $ m_{\rm c}(w,z)$ is the Stieltjes transform of a  density $ \rho_{\rm c} (x,z)$ and   we have
\be\label{mrho}
 m_{\rm c}(w,z)= \int_\R \frac{\rho_{\rm c} (x,z)}{x-w}\rd x
\ee
whenever   $\eta>0$. The function $\rho_{\rm c} (x,z)$ is the limiting eigenvalue density of the matrix $Y^*_z Y_z$
(cf. Lemmas 4.2 and 4.3 in \cite{Bai1997}).
This measure is compactly supported and satisfies  the following bounds. Let
 \be\label{deflampm}
 \lambda_\pm :=\lambda_{\pm }(z):=\frac{( \al\pm3)^3}{8(\al\pm1)} ,\quad \al:=\sqrt{1+8|z|^2}.
\ee
Note  that $\lambda_-$ has the same sign as $|z|-1$.
It is well-known that the density $ \rho_{\rm c} (x,z)$  can be obtained from its Stieltjes transform $m_{\rm c}(x+\ii\eta,z)$ via
$$
 \rho_{\rm c} (x,z)= \frac1\pi\im  \lim_{\eta\to 0^+}m_{\rm c}(x+\ii\eta,z)
=\frac1\pi\1_{x\in[\max\{0,\lambda_-\},\lambda_+]} \im  \lim_{\eta\to 0^+}m_{\rm c}(x+\ii\eta,z).
$$
The propositions and lemmas in this subsection  summarize  the properties of $\rho_{\rm c}$ and $m_{\rm c}$ that we will need in this paper. They contain slight
extensions (to the case $|z|=1+\oo(1)$) of the analogous statements \cite{BouYauYin2012Bulk}. We omit the proof, strictly similar to
the calculation  mentioned in
\cite{BouYauYin2012Bulk}.

In the following, we use the notation $A\sim B$ when $c B \le  A\leq c^{-1}B$,
where $c>0$ is independent of $N$.

\newpage

\begin{wrapfigure}[18]{r}{0.7\textwidth}
\begin{center}
\includegraphics[width=0.7\textwidth]{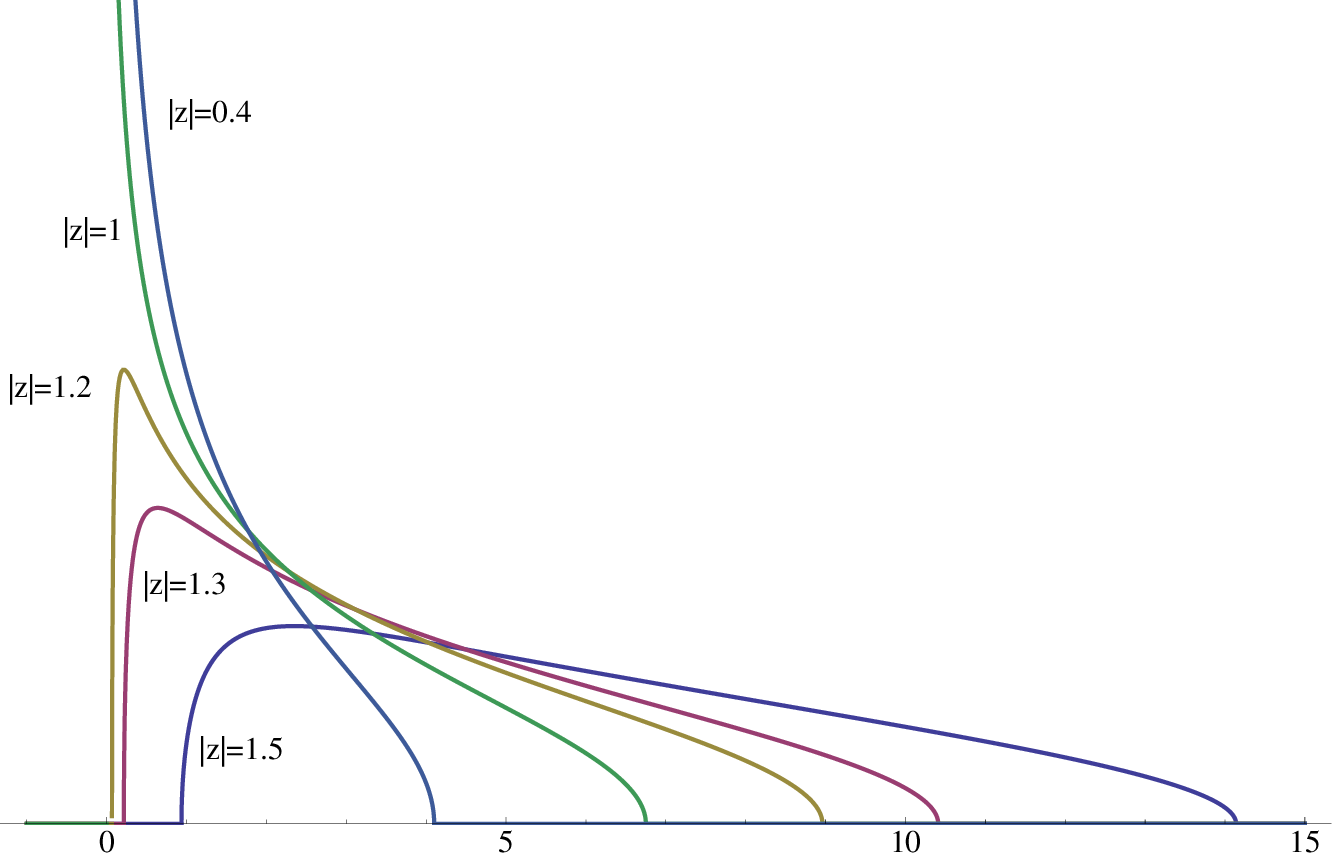}
\end{center}
\label{fig:density}
~\hspace{1cm}The limiting eigenvalues density $\rho_{\rm c}(x,z)$ for $z=1.5, 1.3, 1.2, 1, 0.4$.
\end{wrapfigure}
~

\begin{proposition} \label{prorhoc}
The limiting density $\rho_{\rm c}$ is compactly supported and
the following properties  hold.
\begin{enumerate}[(i)]
\item
The support of $\rho_{\rm c}(x, z)$ is $[\max\{0,\lambda_-\}, \lambda_+]$.
\item
As $x\to \lambda_+$ from below,  the behavior of $\rho_{\rm c}(x, z)$ is given by
$\rho_{\rm c}(x, z)\sim \sqrt {\lambda_+-x}.
$
\item For any fixed $\e>0$,
if $ \max\{0,\lambda_-\}+\e\leq x \leq \lambda_+-\e$, then   $\rho_{\rm c}(x, z)\sim 1$.
\item
Near  $\max\{0,\lambda_-\}$, if  $z$ is  allowed to change with $N$ such that  $|z|= 1+\e_N$ (with $\e_N=\oo(1)$),  then $\lambda_-=\oo(1)$ and  the behavior of $\rho_{\rm c}(x, z)$ depends on $\e_N$.  We will not need this detailed property and will only need the following upper bound: for any $\delta>0$ there is $C>0$ such that for any
$|z|\leq \delta^{-1}$, $x\in[0,\delta]$, we have
$\rho_{\rm c}(x, z)\leq C/ \sqrt x$.
 \end{enumerate}
  \end{proposition}

From this proposition, we obtain the following bound regarding the Stieltjes transform of $\rho_{\rm c}$ (see \eqref{mrho}).

\begin{proposition}\label{tnf}
Uniformly in $z$ and $w$ in any compact set, the Stieltjes transform of $\rho_{\rm c}$ is bounded by
 \be\label{dpjb}
 m_{\rm c}(w,z)=\OO({  |w|}^{-1/2}  ).
 \ee
Moreover, for any fixed $\e>0$, uniformly in $\e<|w|<\e^{-1}$, $|z|<\e^{-1}$,  we have
$m_{\rm c}\sim1$.
\end{proposition}

We now
collect some technical properties of $m_{\rm c}$ used in this paper.
Define   $\kappa:= \kappa (w, z) $ as the distance from $E$ to $\{\lambda_+, \lambda_-\}$:
\be\label{37}
\kappa=\min\{|E-\lambda_-|, |E-\lambda_+|\} .
\ee
For $|z|\le 1$, we have $\lambda_-< 0$, so in this case we define $\kappa:=|E-\lambda_+|$.

 The following two lemmas, concerning the case $c \le |w| \le c^{-1}$ ,  are analogous to Lemma 4.1-4.3 of \cite{BouYauYin2012Bulk}, and the proofs
are essentially the same.
Notice that   the properties of $\rho_c(x)$  used in \cite{BouYauYin2012Bulk} when $c \le |w| \le c^{-1}$ can be summarized as follows:
(1) $\rho_c(x) \sim \sqrt{\lambda_+-x}$ if $0< c \le x\le \lambda_+ $.
(2) $\int_{0}^{\oo(1)} \rho_c(x)=\oo(1)$.
From  Proposition \ref{prorhoc},  these two estimates hold uniformly for all $|z|\le C$, including $|z|-1=\oo(1)$.
We therefore can use the proofs in \cite{BouYauYin2012Bulk} to obtain the following two Lemmas.

  \begin{lemma} \label{pmcc13}
 There exists $\tau_0>0$ such that for  any $\tau \le \tau_0$, if  $||z|-1|\le   \tau $
and  $    \tau \leq |w|\le \tau^{-1}  $
then the  following properties concerning  $m_{\rm c}$ hold (all constants in the following estimates  depending on $\tau$).
Notice that by  \eqref{deflampm}, we have  in this case
 $\lambda_-=\frac{( \al-3)^3}{8(\al-1)} = \OO(||z|-1|^3)=\OO(\tau^3)$.
 \begin{itemize}
\item[Case 1:] $E\geq \lambda_+$ and  $|w-\lambda_+|\ge \tau $. We have
 \be\label{A17n}
|\re m_{\rm c}|\sim 1,\quad0\geq \re m_{\rm c}\geq -\frac12 , \quad \im m_{\rm c}\sim \eta.
 \ee
\item[Case 2:]  $|w-\lambda_+|\le \tau$. We have
\be\label{A18n}
m_{\rm c}(w, z)=- \frac2{3+\al} +  \sqrt{\frac{8(1+\al)^3}{\al(3+\al)^5}}\, (w-\lambda_+ )^{1/2} +\OO(\lambda_+-w), \ee
 and
 \begin{align}\label{esmallfakene}
\im m_{\rm c}\sim & \left\{\begin{array}{cc}
 \frac{\eta}{\sqrt{\kappa}} & \mbox{if  $\kappa\ge\eta$ and $E\ge \lambda_+$,} \\  & \\
\sqrt{\eta} & \mbox{if $\kappa\le \eta$ or $E\le \lambda_+$.}
\end{array}
\right.
\end{align}
\item[Case 3:]     $|w-\lambda_+|\ge  \tau$ and  $E\leq \lambda_+$ (notice that $E $ is allowed to be smaller than $\lambda_-$).
We have
\be\label{A21t}
|m_{\rm c}|\sim 1,\quad \im m_{\rm c}\sim 1.
\ee
\end{itemize}
\end{lemma}

\begin{lemma}\label{pmcc12.5}   There exists $\tau_0>0$ such that for  any $\tau \le \tau_0$   if    the conditions  $||z|-1|\le \tau $ and  $\tau\leq |w|\le \tau^{-1}$ hold,   then  we have
the following three bounds  concerning  $m_{\rm c}$ (all constants in the following estimates  depend on $\tau$):
\be\label{dA2}
|m_{\rm c}+1|\sim |m_{\rm c}|\sim |w|^{-1/2},
\ee
  \be\label{26ssa}
 \left |  \im \frac{1}{ w(1+m_{\rm c})}  \right |  \leq C  \im  m_{\rm c},   \ee
\be\label{ny27}
 \left|(-1 + |z^2|)
   \left( m_{\rm c}-\frac{-2}{3+\al}\right)  \left( m_{\rm c}-\frac{-2}{3-\al}\right)\right|\geq C\frac{\sqrt{\kappa+\eta}}{ |w|}.
\ee
\end{lemma}

\subsection{The Strong Green function estimates. }\label{subsec:Green}
We now state  precisely the estimate regarding  the convergence of  $m$ to $m_{\rm c}$.
Since the matrix $Y^*_z Y_z$ is symmetric, we will follow the approach of \cite{ErdYauYin2010Adv}.
We will use extensively the following definition of high probability events.

\begin{definition}[High probability events]\label{def:hp}
 Define
\be\label{phi}
\varphi\;\deq\; (\log N)^{\log\log N}\,.
\ee
Let $\zeta> 0$.
We say that an $N$-dependent event $\Omega$ holds with \emph{$\zeta$-high probability} if there is some constant $C$ such that
$$
\P(\Omega^c) \;\leq\; N^C \exp(-\varphi^\zeta)
$$
for large enough $N$.
\end{definition}
 For $  \alpha \geq 0$,  define the $z$-dependent  set
$$
\b S(\alpha)\;\deq\; \hb{w \in \C \st  \max (\lambda_-/5, 0) \le   E \leq 5\lambda_+ \,,\;     \varphi^\alpha N^{-1} |m_{\rm c}|^{-1}\leq \eta \leq 10 },
$$
 where $\varphi$ is defined in \eqref{phi}.  Here  we have  suppressed the explicit $z$-dependence.

 \begin{theorem}  [Strong local Green function estimates] \label{sempl}
 Let $\e>0$ be given. Suppose
that $z$ is bounded in $\mathbb{C}$, uniformly in $N$.
Then  for any $\zeta>0$, there exists $C_\zeta>0$ such that  the following event  holds  \hp{\zeta}:
 \be\label{res:sempl}
\bigcap_{w \in{ {\b S}}(C_\zeta),|w|>\e} \hbb{|m(w)-m_{\rm c}(w)|  \leq \varphi^{C_\zeta} \frac{1}{N\eta}}.
 \ee
Moreover,
the individual  matrix elements of
the Green function  satisfy, \hp{\zeta},
\begin{align}\label{Lambdaofinal}
\bigcap_{w \in{ {\b S}}(C_\zeta),|w|>\e} \hbb{\max_{ij}\left|G_{ij}-m_{\rm c}\delta_{ij}\right| \leq \varphi^{C_\zeta} \left(\sqrt{\frac{\im\,  m_{\rm c}  }{N\eta}}+ \frac{1}{N\eta}\right)}.
\end{align}
\end{theorem}

\begin{proof}
The proof mimics the one of Theorem 3.4 in \cite{BouYauYin2012Bulk}, in  the case of general $z$ but restricted to $|w|\ge \e$. 
Indeed, in the special case $|w|\ge \e$, Theorem 3.4 in \cite{BouYauYin2012Bulk} (which is stated for$||z|-1|>0$) actually also applies for $|z|=1$:
its proof only needs
some properties of $m_c$ in this case, which are Lemmas 4.1 and 4.2 in \cite{BouYauYin2012Bulk}, and have as analogues (even for $|z|=1$) Lemmas \ref{pmcc13} and \ref{pmcc12.5}
from the previous subsection.
\end{proof}

\subsection{Conclusion. }
 Following the Hermitization explained in the previous section, in order to prove Theorem \ref{z1}, we need to properly bound (\ref{id2}).
Now $\log \lambda_j(z)$ needs to be evaluated even for $z$ close to the edge, and for this we first define a proper approximation of the logarithm, with a
cutoff at scale $N^{-2+ \e}$.

\begin{definition}\label{defphiY}
Let $h(x)$ be a smooth  increasing  function supported on $[1,+\infty]$  with  $h(x)=1$
for $x\geq 2 $ and $h(x)=0$  for $x\leq 1$.
For any $\e>0$, define $\phi$ on $\R_+$ by
\be\label{defvarphi}
 \phi(x)=\phi_\e(x)= h(N^{2-2\e} x)\, (\log x)\,\left(1- h\left(\frac{x}{2\lambda_+}\right)\right).
 \ee
 \end{definition}

We now prove that $\phi$ is a proper approximation for $\log$ in the sense that
\be\label{jatb}
\tr \left( \log (Y^* Y ) -  \phi(Y^* Y ) \right) \prec  N^{C\e}
\ee
uniformly in $z$ in any compact set, where  $C$ is an absolute constant independent of $N$.

For this purpose,  we first bound the number of small eigenvalues of $Y^* Y$. This
is the aim of the following lemma,
proved in Subsection \ref{sec:upperBound}.

\begin{lemma} \label{a priori}
Suppose that $|w|+ |z| \le M$ for some constant $M$ independent of $N$.
Under the same assumptions as Theorem \ref{z1T},
for any $\zeta>0$, there exists $C_\zeta$  such that
if
 \be\label{eta}
 \eta \geq \varphi^{C_\zeta} N^{-1}|w|^{1/2}
 \ee
then we have
 \be\label{52s}
 |m(w,z)|\leq (\log N) |w|^{-1/2}
 \ee
 \hp{\zeta} (notice that if we take  $\eta \sim |w|$ then the restriction is
 $|w|\ge N^{-2}\varphi^{2C_\zeta}$).
\end{lemma}

Using
\be\label{nkwS}
\left|\{j: \lambda_j\in [E-\eta, E+\eta]\}\right|\leq C \eta \im m(E+\ii\eta)
\ee
and Lemma \ref{a priori} with $E=0$ and $\eta =N^{-2+2\e}$, we obtain
\be\label{nkwT}
\frac{1}{N}\left|\{j: |\lambda_j|\leq N^{-2+2\e}\}\right|\prec N^{\e}.
 \ee
Moreover, we have  the following lower bound on the smallest eigenvalue.

\begin{lemma}\label{smallest}
Under the same assumptions as Theorem \ref{z1T},
$$
 | \log \lambda_1(z) |  \prec 1
$$
holds uniformly for $z$
in any fixed compact set.
\end{lemma}

\begin{proof}
This lemma follows from
\cite{RudVer2008} or  Theorem 2.1 of \cite{TaoVu2008}, which gives the required estimate uniformly in $z$. Note that
the  typical size of $\lambda_1$ is $N^{-2}$ \cite{RudVer2008}, and
we need a much weaker bound of type $\Prob(\lambda_1(z)\leq e^{-N^{-\e}})\leq N^{-C}$ for any $\e,C>0$.
This estimate is very  simple to prove if, for example, the entries of $X$ have a density bounded by $N^C$, which was discussed in Lemma 5.2 of \cite{BouYauYin2012Bulk}.
\end{proof}

The preceding lemma together with
(\ref{nkwT}) implies
\be\label{jatb2}
\tr \left( \log (Y^* Y ) -  \tilde\phi(Y^* Y ) \right) \prec  N^{C\e},
\ee
where $\tilde \phi(x)=h(N^{2-2\e} x)(\log x)$.
The proof of (\ref{jatb}) is then complete thanks to the following
easy lemma bounding the contribution of exceptionally large eigenvalues.

\begin{lemma}\label{largest}
Under the same assumptions as Theorem \ref{z1T}, for any $\e>0$,
\be
\tr \left( \phi(Y^* Y ) -  \tilde\phi(Y^* Y ) \right) \prec  1
\ee
holds uniformly for $z$
in any fixed compact set.
\end{lemma}

\begin{proof}
Notice first that, for any $c_0>0$, there is a constant $C>0$ such that, uniformly in $|z|<c_0$,
$\lambda_N\prec N^C$ (this is for example an elementary consequence of the inequality $\lambda_N\leq \tr(Y^*Y)$).
The proof of the lemma will therefore be complete if
\begin{equation}\label{eqn:sufficient}
\left|\{{ i}:\lambda_i\geq 2\lambda_+\}\right|\prec 1.
\end{equation}
Let $(\lambda^{(\delta)}_k)_{k\in\llbracket 1,N\rrbracket}$ be the increasing eigenvalues of
${Y^{(\delta)}}^*Y^{(\delta)}$, where $Y^{(\delta)}=X-(1-\delta)z$. By the
Weyl inequality for the singular values, for any $\delta>0$
\begin{equation}\label{eqn:Weyl}
\lambda_k^{1/2}\leq {\lambda^{(\delta)}_k}^{1/2}+ \delta.
\end{equation}
Moreover, as $|(1-\delta)z|<1$, we are in the context of Lemma 5.1 in \cite{BouYauYin2012Bulk}, which yields,
for any $\zeta>0$, the existence of $C_\zeta>0$ such that
\begin{equation}\label{eqn:conc}
\lambda^{(\delta)}_{N-\varphi^{C_\zeta}}\leq\lambda^{(\delta)}_+
\end{equation}
with $\zeta$-high probability. Equations (\ref{eqn:Weyl}) and (\ref{eqn:conc}) give
$$
\lambda_{N-\varphi^{C_\zeta}}\leq\lambda_+^{(\delta)}+2\delta{\lambda_+^{(\delta)}}^{1/2}+\delta^2.
$$
Together with $\lambda^{(\delta)}_+\to \lambda_+$ as $\delta\to 0$ (uniformly in $|z|<c_0$), this concludes the proof of
(\ref{eqn:sufficient}) by choosing $\delta$ small enough.
\end{proof}

\begin{remark}
With some more technical efforts, the techniques used in this paper allow to prove the stronger result
$
\lambda_N-\lambda_+\prec N^{-2/3}
$. This rigidity on the edge could be proved similarly to the rigidity of the
largest eigenvalues of $X^*X$, for non-square matrices \cite{PilYin2011}.
\end{remark}

It is known, by  Lemma 4.4 of \cite{Bai1997}, that
 \begin{align}\label{28jT}
\int_0^\infty(\log x)\Delta_z\rho_{\rm c}(x,z) \rd x =4 \chi_D (z).
\end{align}
By definition of $\phi$ and $\rho_c$
\be
\left|\int\phi(x)\rho_{\rm c}(x)\rd x -\int\log (x)\rho_{\rm c}(x)\rd x \right|\leq N^{-1+C\e}
\ee
Then from equations (\ref{id2}) and (\ref{jatb}),  Theorem \ref{z1}  follows if we can prove
estimate
\be\label{mjs}
\left| \int \Delta f(\xi) \left( \tr \phi(Y^* Y ) {  -N\int\phi(x)\rho_{\rm c}(x)\rd x}\right) \rd\xi \rd \bar \xi \right|\prec N^{ C\e}.
\ee
We first evaluate  $\tr \phi(Y^* Y )$ in the following lemma.

\begin{lemma}\label{lem:ben}
Let $ \chi$  be a smooth cutoff function  supported in $ [-1,1]$ with  bounded derivatives and $ \chi(y) = 1$ for $|y| \le 1/2$.
For the same $\e$ presented in the definition of $\phi_\e$, we define the domain in $\C$
\be\label{defI1I2}
I=\left\{ N^{-1+\e}\sqrt E\leq \eta, |w|\leq \e\right\}.
\ee
There is some $C>0$ such that for any $ \e>0$, {  (recall: $\phi:=\phi_\e$  from Def. \ref{defphiY}})
\be\label{yxl}
 \left|\tr \phi(Y^* Y)-N\int\phi(x)\rho_{\rm c}(x)\rd x- \frac{N}{\pi} \int_{I} \chi(\eta) \phi' (E)
  \re \left(m(w)-m_{\rm c}(w)\right)    \rd E\rd \eta \right|\prec N^{C\e  }
\ee
uniformly in $z$ in any compact set.
 \end{lemma}

\begin{proof} Suppose  that \eqref{yxl} holds with the integration range $I$ replaced by  $\tilde I$  where
$$
\tilde I=\left\{ N^{-1+\e}\sqrt E\leq \eta\right\}.
$$
Define  the notation  $\Delta m=m-m_{\rm c}$.
One can easily check that if $w\in  \tilde I$
and $\phi'(E)\neq 0 $, then
$
\eta\geq N^{-1+\e}|w|^{1/2}$ (if $\eta\leq E$ this last equation exactly means $\omega\in\tilde I$ and if $\eta\geq E$ it is equivalent to $\eta\geq N^{-2+2\e}$, which is true because $\eta\geq N^{-1+\e}\sqrt E$ and $E\geq N^{-2+2\e}$ when $\phi'(E)\neq 0$).

Using Proposition \ref{tnf} and Theorem \ref{sempl}, we therefore have $\re \left(\Delta m(w)\right)=\OO(\varphi^{C_\zeta}/(N\eta))$ on $\tilde I\slash I$, so
$$
\int_{ \tilde I\slash I} \chi(\eta) \phi' (E)
  \re \left(\Delta m(w)\right)\rd E\rd\eta
\prec   \int_{  \e \le     |w|\le    C , N^{-1+\e}\sqrt{E}\leq\eta\leq C}\frac{\rd E\rd\eta}{NE\eta}\prec N^{-1}.
$$
Hence \eqref{yxl} holds with integration range given by $I$. \nc

We now prove \eqref{yxl} with $I$ replaced by $\tilde I$.
The Helffer-Sj\"ostrand functional calculus (see e.g. \cite{Dav1995}), gives for any smooth function $q: \R\to \R$,
$$
q(\lambda)=\frac{1}{2\pi}\int_{\R^2}\frac{\ii y q''(x)\chi(y)+\ii (q(x)+\ii y q'(x))\chi'(y)}{\lambda-x-\ii y}\rd x\rd y.
$$
As a consequence,
\begin{align}
\tr    \phi (Y^* Y) &=\frac N{\pi}\int_{\eta>0}\label{113th}
\Big( \ii \eta \phi ''(E)\chi (\eta)+ \ii \phi (E) \chi'(\eta)-\eta
\phi '(E)\chi'(\eta)\Big)m(E +\ii\eta)
  \dd E \dd \eta,\\
N\int\phi(x)\rho_{\rm c}(x)\rd x &= \frac N{\pi}\int_{\eta>0}\label{113thj}
\Big( \ii \eta \phi ''(E)\chi (\eta)+ \ii \phi (E) \chi'(\eta)-\eta
\phi '(E)\chi'(\eta)\Big)m_{\rm c}(E +\ii\eta)
  \dd E  \dd \eta.
\end{align}
By definition,   $\chi'(\eta)\neq 0$ only if  $|\eta|\in[1/2,1]$. Using  \eqref{res:sempl},  we have
$$
\frac N{\pi}   \int_{\eta>0}
\Big(   \ii \phi (E) \chi'(\eta)-\eta
\phi '(E)\chi'(\eta)\Big)\Delta m(E+\ii\eta)
  \dd E  \dd \eta  \prec 1.
$$
hence  the difference between \eqref{113th} and \eqref{113thj} is
$$
 \tr \phi(Y^* Y)-N\int\phi(x)\rho_{\rm c}(x)\rd x =  \frac N{\pi}\int_{\eta>0}
  \ii \eta \phi ''(E)\chi (\eta)  \Delta m(E+\ii\eta)
  \dd E \, \dd \eta +\OO_\prec(1).
$$
Since  the left  side of the equation,  $\chi$ and $\phi$ are all real,  this
can be rewritten as
$$
 \tr \phi(Y^* Y)-N\int\phi(x)\rho_{\rm c}(x)\rd x =  -\frac N{\pi}\int_{\eta>0}
  \eta \phi ''(E)\chi (\eta)  \im \Delta m(E+\ii\eta)
  \dd E \dd \eta +\OO_\prec(1).
$$

Furthermore, for  $\eta\leq   N^{-1+\e}\sqrt{E}$ and $E\geq N^{-2+2\e}$,
we have $
\eta \im m(E+\ii \eta)\prec N^{-1}
$, because
$\eta\mapsto\eta  \im m(E+\ii\eta)$ is increasing, and at
$\eta_0 = N^{-1+\e}\sqrt{E}$ the result its true thanks to  \eqref{52s}
($\eta\leq   N^{-1+\e}\sqrt{E}$ and $E\geq N^{-2+2\e}$ imply that $\eta>N^{-1+\e}|w|^{1/2}$ as we saw previously).
Together with the easy bound $|\phi''(E)|\leq C(1+|\log(E)|) E^{-2}$, the above estimate yields
\be\label{ald}
 N
\int_{\tilde I^c}  \eta \phi''(E) \chi(\eta)
\im m(w)\rd E\rd \eta\prec \int_{N^{-2}\leq E\leq \OO(1)} C N^{-1+C\e}E^{-3/2}\rd E\prec  N^{C\e}.
\ee
The same inequality holds when replacing  $m$ with $m_{\rm c}$, hence
$$
 N
\int_{\tilde I^c}  \eta \phi''(E) \chi(\eta)
\im (\Delta m(w))\rd E\rd \eta\prec  N^{C\e}.
$$
 To estimate
$
 N
\int_{\tilde I} \eta \phi''(E) \chi(\eta)
\im (\Delta m(w))\rd E\rd \eta,
$
we integrate this term by parts first in $E$, then in $\eta$
(and use the Cauchy-Riemann equation $\frac{\partial}{\partial E}\im (\Delta m)=-\frac{\partial}{\partial \eta}
\re(\Delta m)$) so that
 \begin{align}\label{q1}
 N
\int_{\tilde I} \eta \phi''(E) \chi(\eta)
\im (\Delta m(w))\rd E\rd \eta=
 & N \int   \eta\chi(\eta)
  \phi' (N^{2-2\e}\eta^2)\im (\Delta m(N^{2-2\e}\eta^2+\ii\eta))   \rd \eta  \non
 \\
&+ N  \int  N^{-1+\e}\sqrt E  \phi' (E)
\chi(  N^{-1+\e}\sqrt E)\re (\Delta m(E+\ii N^{-1+\e}\sqrt E))\rd E  \non
\\
 &- N
\int_{\tilde I}\chi(\eta)  \phi' (E)
\re(\Delta m(w))\rd E\rd \eta\non  \\
&- N
\int_{\tilde I}\eta\chi'(\eta)  \phi' (E)
\re(\Delta m(w))\rd E\rd \eta.
   \end{align}
The last term is nonzero only if $|\eta|\in[1/2,1]$.
 By  \eqref{res:sempl},
this term is of order  $\OO_\prec(1)$.

Concerning the first two terms, they can be bounded in the following way:
as we already saw, if $w\in  \tilde I$
and $\phi(E)\neq 0 $, then
$
\eta\geq N^{-1+\e}|w|^{1/2}
$;
one can therefore use (\ref{52s}):
$|m|+|m_c|\leq (\log N)|w|^{-1/2}$. This together with $|\phi'(E)|\leq C (1+|\log (E)|) E^{-1}$
proves that the first two terms are
$\OO_\prec( N^{C\e}) $,
completing the proof: only the third term remains.
\end{proof}

  To
complete the proof of Theorem \ref{z1}, we estimate the third term in \eqref{q1}
under the additional  assumption that  the  third moments of the matrix entries vanish.
The following lemma provides such an  estimate whose proof will be postponed to Section \ref{sec:wfgb}.
  To state this lemma, we introduce the following notation,
$$
Z_{X_1,X_2}^{(g)}=N
\int  \Delta g(\xi)  \int_{I}  \chi(\eta)\phi '(E)
\re (m_1(w)-m_2(w)) \rd E\rd \eta\rdA(\xi),$$
for any given smooth function $g$ and initial random matrices $X_1,X_2$,
with associated matrices ${Y^{(1)}_z}^* Y^{(1)}_z$, ${Y^{(2)}_z}^* Y^{(2)}_z$ having
respective Stieltjes transforms $m_1$ and $m_2$ (the parameter $z$ and $\xi$ are related by (\ref{eqn:zXi})).
Moreover, we will write $$Z_{X,{\rm c}}^{(g)}=N
\int  \Delta g(\xi)  \int_{I}  \chi(\eta)\phi '(E)
\re (m(w)-m_{\rm c}(w)) \rd E\rd \eta\rdA(\xi).$$
{Notice that $Z_{X_1,X_2}^{(g)}$ and $Z_{X,{\rm c}}^{(g)}$  depend on $\e$ through the definition of $\phi$ in Def. \ref{defphiY}.}

\begin{lemma}\label{wfgb} Under the {assumptions} of Theorem \ref{z1}, there exists a constant   $C>0$ such that
for any small enough $\e>0$, we  have
$$
Z_{X,{\rm c}}^{(f)}\prec N^{C\e}c_f,
$$
where $c_f$ is a constant depending only on the function $f$.
\end{lemma}

Combining this lemma with \eqref{yxl}, we obtain \eqref{mjs} and complete  the proof of Theorem \ref{z1}.

\medskip

To prove Theorem \ref{z1T}, instead of Lemma \ref{wfgb},  we only need to prove the following lemma which
does not assume the vanishing third moment condition.  This lemma will be proved at the end of Section 3.
This concludes the proof of Theorem \ref{z1T}.

\begin{lemma}\label{wfgbT} Under the assumptions  of Theorem \ref{z1T}, for some fixed $C>0$, for any small enough $\e>0$, we  have
$$
Z_{X,{\rm c}}^{(f)}\prec N^{1/2+C\e}c_f,
$$
where $c_f$ is a constant depending only on the function $f$.
\end{lemma}

\section{Proof of Lemmas \ref{a priori} and \ref{wfgb}}

\subsection{Preliminary lemmas. } This subsection summarizes some elementary results
from \cite{BouYauYin2012Bulk}, based on large deviation estimates.
Note that all the inequalities in this subsection hold uniformly for bounded $z$,
no matter its distance to the unit circle.
We first introduce some notations.

\begin{definition} \label{definition of minor}
Let $\bb T, \bb U \subset \llbracket 1,N\rrbracket$. Then we define $Y^{(\bb T, \bb U)}$ as the $ (N-|\bb U|)\times (  N-|\bb T|)  $
matrix obtained by removing all columns of $Y$ indexed by $i \in \bb T$ and all rows  of $Y$ indexed by $i \in \bb U$. Notice that we keep the labels of
indices of $Y$ when defining $Y^{(\bb T, \bb U)}$.

Let $\by_i$ be the $i $-th column of $Y$ and  $\by^{(\bb S)}_i$ be the vector obtained by removing $\by_i (j) $ for
 all  $ j \in  \bb S$. Similarly we define $\mathrm y_i$ be the $i $-th row of $Y$.
Define
\begin{align*}
G^{(\bb T, \bb U)}=\Big [  (Y^{(\bb T, \bb U)})^* Y^{(\bb T, \bb U)}- w\Big]^{-1},\ \  & m_G^{(\bb T, \bb U)} =\frac{1}{N}\tr G^{(\bb T, \bb U)},
\\
\mG^{(\bb T, \bb U)}= \Big [ Y^{(\bb T, \bb U)}(Y^{(\bb T, \bb U)})^*- w \Big ]^{-1},\ \ & m_\mG^{(\bb T, \bb U)} =\frac1N\tr \mG^{(\bb T, \bb U)}.
\end{align*}
By definition,  $m^{(\emptyset, \emptyset)} = m$.
 Since the eigenvalues of $Y^* Y $ and $Y Y^*$ are the same except the zero eigenvalue, it is easy to check that
\be\label{35bd}
m_G^{(\bb T, \bb U)}(w) =m_\mG^{(\bb T, \bb U)} +\frac{|\bb U|-|\bb T|}{N  w}
\ee
For $|\bb U|=| \bb T|$, we define
\be\label{d trGmG}
m ^{(\bb T, \bb U)}:= m_G^{(\bb T, \bb U)} = m_\mG^{(\bb T, \bb U)}
\ee
\end{definition}

  \begin{lemma} [Relation between $G$, $G^{(\bb T,\emptyset)}$ and  $G^{( \emptyset, \bb T)}$]  \label{lem: GmG}
 For $i,j \neq k  $ ( $i = j$ is allowed) we have
\be\label{111}
 G_{ij}^{(k,\emptyset)}=G_{ij}-\frac{G_{ik}G_{kj}}{G_{kk}}
,\quad
\mG_{ij}^{(\emptyset,k)}=\mG_{ij}-\frac{\mG_{ik}\mG_{kj}}{\mG_{kk}},
\ee
\be\label{Gik}
 G^{ ( \emptyset,i)}
= G+\frac{(G  {\mathrm y} _i^*) \, ( {\mathrm y} _i  G)}
{1-  {\mathrm y} _i G    {\mathrm y} _i ^*}
,\quad
G
 =G^{ ( \emptyset,i)}-\frac{( G^{ ( \emptyset,i)} {\mathrm y} _i^*)  \,
 (  {\mathrm y} _i  G^{ ( \emptyset, i)})}
 {1+   {\mathrm y} _i  G^{ ( \emptyset,i)} {\mathrm y} _i ^* },
 \ee
and
$$
\mG^{ (i,\emptyset)}
=\mG+\frac{(\mG  \by_i) \, (\by_i^* \mG)}
{1-   \by_i^*   \mG     \by_i  }
,\quad
 \mG
 =\mG^{ (i,\emptyset)}-\frac{(\mG^{ (i,\emptyset)}  \by_i)  \,
 ( { \by_i}^ * \mG^{ (i,\emptyset)})}
 {1+  \by_i^*\mG^{ (i,\emptyset)}   \by_i  }.
$$

Furthermore, the following crude  bound on the difference between $m$ and $m_G^{(U, T)}$ holds:
for $\bb U,  \bb T\subset \llbracket 1,N\rrbracket$    we   have
\be\label{37mk}
|m-m^{(\bb U, \bb T)}_G|+|m-m^{(\bb U, \bb T)}_\mG| \leq   \frac{|\bb U|+|\bb T|}{N\eta}.\quad
 \ee
   \end{lemma}

 \begin{definition}\label{Zi-def}
 In the following, $\E_X$ means the integration with respect to the random variable $X$.
For any $\bb T\subset \llbracket 1,N\rrbracket$, we introduce the   notations
$$
Z^{(\bb T)}_{i }:=(1-\E_{{\mathrm y}_i})
 {\mathrm y}^{(\bb T)}_i  G^{(\bb T, i)} {\mathrm y}_i^{(\bb T)*}
$$
and
$$
\cal Z^{(\bb T)}_{i }:=(1-\E_{\by_i})
\by_i^{(\bb T) *} \mG^{(i, \bb T)} \by_i^{(\bb T)}.
$$
 Recall by our convention that
$\by_i$ is a $N\times 1 $ column vector and $\mathrm y_i$ is a $1\times N $ row vector.
For simplicity we will write
$$
Z _{i }
=Z^ {(\emptyset)}_{i}, \quad \cal Z _{i }
=\cal Z^ {(\emptyset)}_{i}.
$$
\end{definition}

 \begin{lemma}  [Identities for  $G$, $\mG$, $Z$ and  $\cal Z$]   \label{idm}
 For any $ \T\subset \llbracket 1,N\rrbracket$, we have
\begin{align}\label{110}
 G^{(\emptyset , \bb T)} _{ii}
 & =   - w^{-1}\left[1+  m_\mG^{(i, \bb T)}+   |z|^2 \mG_{ii}^{(i, \bb T)} +\cal Z^{(\bb T)}_{i } \right]^{-1},
\end{align}
\be\label{110b}
 {G_{ij} ^{(\emptyset , \bb T) } }
  =   -wG_{ii}^{(\emptyset , \bb T) } G^{(i,\bb T)}_{jj}
\left(   \by_i^{(\bb T)*}  \mG^{(ij, \bb  T)}  \by_j^{(\bb T)}\right) , \quad i\neq j,
\ee
where, by definition,  $\mG_{ii}^{(i,\bb T)}=0$ if $i\in \bb T$.  Similar results hold for $\mG$:
\be\label{110c}
\left[\mG^{(\bb T, \emptyset)} _{ii} \right]^{-1}
  =   - w\left[1+  m_ G^{(\bb  T,i)}+   |z|^2  G_{ii}^{(\bb T,i)} + Z^{(\bb T)}_{i } \right]
\ee
\be\label{110d}
 {\mG_{ij}^{(\bb T, \emptyset)}}
  =   -w\mG_{ii}^{(\bb T, \emptyset)}\mG^{(\bb T, i)}_{jj}\left( \mathrm y_i^{(\bb T)} G^{( \bb T,ij)}   \mathrm y_j^{(\bb T)*}\right), \quad i\neq j.
\ee
\end{lemma}

\begin{lemma}[Large deviation estimate]\label{lem:bh} For any $\zeta>0$, there exists $Q_\zeta>0$ such that  for $\bb T\subset\llbracket 1,N\rrbracket$, $|\bb T| \leq N/2$ the following estimates hold  \hp{\zeta}:
\begin{align}\label{130}
| Z^{(\bb T)}_{i }|=
\left|(1-\E_{ \mathrm y_i})  \left(\mathrm y_i^{(\bb T)} G^{(\bb T,i )}    \mathrm y_i^{(\bb T)*}\right)  \right|
 \leq  \varphi^{  Q_\zeta/2} \sqrt{\frac{\im m_ G^{(\bb T,i )}+ |z|^2 \im G^{(\bb T,i )}_{ii} }{N\eta}}, \\
|\cal Z^{(\bb T)}_{i }| =
 \left|(1-\E_{ \by_i})   \left(\by_i^{(\bb T)*}  \mG^{(i , \bb T )}  \by_i^{(\bb T)} \right) \right|
\leq  \varphi^{  Q_\zeta/2} \sqrt{\frac{\im m_\mG^{(i,\bb T)}+ |z|^2 \im\mG^{(i, \bb T)}_{ii} }{N\eta}}. \non
 \end{align}
 Furthermore, for $i\neq j$, we have
\begin{align}\label{132}
\left|
(1-\E_{\mathrm y_i\mathrm y_j})
\left( \mathrm y_i^{(\bb T)} G^{(\bb T,ij)}   \mathrm y_j^{(\bb T)*}\right)
 \right|
& \leq
  \varphi^{  Q_\zeta/2}
  \sqrt{\frac{\im m_ G^{(\bb T,ij)}
  +|z|^2\im G^{(\bb T,ij)}_{ii}+|z|^2 \im G^{(\bb T,ij)}_{jj}}{N\eta}},
  \\
  \left|
(1-\E_{\by_i\by_j} )
\left(\by_i^{(\bb T)*}   \mG^{(ij,\bb T)}  \by_j^{(\bb T)}\right)
 \right|
& \leq
  \varphi^{  Q_\zeta/2}
  \sqrt{\frac{\im m_\mG^{(ij,\bb T)}
  +|z|^2\im\mG^{(ij,\bb T)}_{ii}+|z|^2 \im\mG^{(ij,\bb T)}_{jj}}{N\eta}}, \label{1321}
 \end{align}
 where
 \be\label{1328}
 \E_{\mathrm y_i\mathrm y_j} \left( \mathrm y_i^{(\bb T)} G^{(\bb T,ij)}   \mathrm y_j^{(\bb T)*}\right)
= |z|^2G^{(\bb T,ij)}_{ij}, \quad
\E_{\by_i\by_j}\left(\by_i^{(\bb T)*}   \mG^{(ij,\bb T)}  \by_j^{(\bb T)}\right)= |z|^2\mG^{(ij,\bb T)}_{ij}.
 \ee
 \end{lemma}

\subsection{Proof of Lemma \ref{a priori}. }\label{sec:upperBound}
 By
 \eqref{130} and \eqref{37mk}, \hp{\zeta}, we have for some large constant $C$
\begin{align}
|Z_i^{(i)}|
 &  \leq
 \varphi^{C_\zeta} \sqrt{\frac{\im m^{(i, i)} }{N\eta}}
  \leq
  \varphi^{C_\zeta}\frac{1}{N\eta}+ \varphi^{C_\zeta}\sqrt{\frac{ | \im m| }{N\eta}}  \nonumber \\ & \le  \varphi^{3C_\zeta}\frac{1}{N\eta}+\frac {| \im m|} { (\log N)^C},
  \label{r2}
\end{align}
where we have used the Cauchy-Schwarz inequality in the last step.
Moreover, choosing
$\bT = \{i\}$ in  \eqref{110c}, we have
\be\label{110c1}
\left[\mG^{(i, \emptyset)} _{ii} \right]^{-1}
  =   - w\left(1+  m_ G^{( i ,i)}+Z^{(i)}_{i }   \right).
\ee
If
\be\label{e}
\varphi^{3C}\frac{1}{N\eta} \ll \frac {|m|} { (\log N)^C } \quad {\rm and}\quad   \quad |m| \ge 3,
\ee
then by   \eqref{110c1}  and \eqref{37mk}
we have
\be\label{a1}
\left | \mG^{(i, \emptyset)} _{ii}   \right |^{-1}
  \ge   |w|\left |1+  m_ G^{( i ,i)}+Z^{(i)}_{i }   \right| \ge  \frac  2  3 |w|\left |  m - \frac {|m|} { (\log N)^C }   \right|
  \ge  \frac 1 2  |w| |  m  |
\ee
Similarly,  \hp{\zeta}, we can bound $\cal Z_i$ by
\be  \label{r1}
|\cal Z_{i }|
\leq   \varphi^{C_\zeta} \sqrt{\frac{\im m_\mG^{(i, \emptyset)}+ |z|^2 \im\mG^{(i, \emptyset)}_{ii} }{N\eta}}
\leq  \varphi^{C_\zeta}\frac{1}{N\eta} +  \frac {  \im m + |z|^2   \im\mG^{(i, \emptyset)}_{ii} } {(\log N)^C}
\ee
If \eqref{52s} is violated, i.e., $  |m|\geq (\log N)|w|^{-1/2}$, then  \eqref{e} is implied by the assumption \eqref{eta}. Hence
\eqref{a1} holds. By  \eqref{110},
$$
|w G_{ii}|^{-1} =  |1+  m_\mG^{(i, \emptyset)}+   |z|^2 \mG_{ii}^{(i,\emptyset)} +  \cal Z_{i } | \ge  | m|/2  - C   - |w|^{-1/2}
$$
and we conclude that  $ |G_{ii}|\leq |w|^{-1/2}$  \hp{\zeta}, for any $1\leq i\leq N$.
This  contradicts the assumption $  |m|\geq (\log N)|w|^{-1/2}$  and  we have thus proved \eqref{52s}.

\subsection{Proof of Lemma \ref{wfgb}. }  \label{sec:wfgb}

We begin with the following estimates on the elements of the Green function, weaker than
those in Theorem \ref{sempl} but valid for any $w$ and $z$.

\begin{lemma}\label{a priori2}
Suppose  $ |w|+ |z| \leq C. $
   For any $\zeta>0$, there exists $C_\zeta$  such that
if the assumption \eqref{eta} holds
then the following estimates hold:
 \be\label{53s}
\max_i|G_{ii}|\leq 2(\log N) |w|^{-1/2},
\ee
\be\label{53.5s}
\max_i   |w | |G_{ii}||\mG^{(i,\emptyset )}_{ii}|\leq  (\log N)^4, \ee
 \be\label{54s}
 \max_{ij}|G_{ij}|\leq C(\log N)^2 |w|^{-1/2},
\ee
\hp{\zeta}. Furthermore, for  eigenvectors $u_\al$ of $Y^* Y$ are delocalized:
\hp{\zeta}
\be\label{52a}
{ \max_\alpha \|u_\al \|^2_\infty} \leq   \varphi^{C_\zeta} N^{-1}.
\ee
\end{lemma}

\begin{proof}
 From
 equations (6.37) and (6.38) of \cite{BouYauYin2012Bulk}, the following estimate holds\footnote{Although these bounds in \cite{BouYauYin2012Bulk} are stated under the assumption $||z|-1|\geq \tau$, the  same argument holds also for $z$ close to the unit circle, under the extra assumption $|w|>\e$.}:
\begin{equation}\label{110a1}
 \mG_{ii}^{(i,\emptyset)}
= - \frac 1 { w( 1+  m  )}  +  \cal E_1
\end{equation}
where
\begin{align}\label{457}
 \cal E _1 &=    w^{-1}\frac 1 { (1+  m)^2  }\Big [  m_ G^{( i ,i)}- m + Z^{(i)}_{i }    \Big ]  +  \OO\left ( \frac {|Z^{(i)}_{i }|^2 + \frac 1 { (N \eta)^2} } { |w| |1+  m|^3  } \right )  = \OO(\varphi^{ Q_\zeta/2} \Psi)\\
 \Psi&= \left(\sqrt{\frac{\im m_{\rm c}+|m-m_{\rm c}|}{N\eta}}+\frac{1}{N\eta}\right).\notag
\end{align}
By using \eqref{res:sempl}-\eqref{Lambdaofinal}, \eqref{110a1}, \eqref{457} and \eqref{dA2}, for $|w|\ge \e$,  one can be easily prove:  $|G_{ij}|+|\mG^{(i,\emptyset )}_{ii}|\leq |w|^{-1/2}$,  which implies  \eqref{53s}-\eqref{54s}.  Moreover,  if  $|z |\le1/2 $,  \eqref{res:sempl}-\eqref{Lambdaofinal} still hold without the restriction $|w|\le \e$ (see Theorem 3.4 in \cite{BouYauYin2012Bulk}). It implies \eqref{53s}-\eqref{54s}  in the case $|z|\leq 1/2$.   From now on, we therefore assume that for some small enough $\e>0$,
$$|w|\leq  \e,\quad |z|\geq 1/2.$$

From \eqref{110},  \eqref{37mk}  and \eqref{110c}, we  have
\be\nonumber
\left|G^{-1}_{ii}-G^{-1}_{jj}\right| \le
    \frac {C |w|} { N \eta} +        W_{ij} +   w |\cal Z_{i }| + w |\cal Z_j| ,
\ee
where

\be\label{50}
     W_{ij}= |z|^2  \left(
   (1+  m_ G^{( i ,i)}+Z^{(i)}_{i } )^{-1}- (1+  m_ G^{( j , j)}+Z^{(j)}_{j })^{-1}  \right).
\ee
 If $\e^{ 1/4}|w|^{-1/2}\leq   |m| \leq (\log N) |w|^{-1/2}$
then by  \eqref{r2} we have  $ |Z^{(i)}_{i } | \leq (\log N)^{-C} |  m|$ due to  the restriction of $\eta$ in \eqref{eta}. Thus  we can estimate  $W_{ij}$
by
\begin{align}\label{51}
|W_{ij} |  & \le  |z|^2  \left [  |Z^{(i)}_{i }| +  |Z^{(j)}_{j }|  +  (N \eta)^{-1} \right ]
   (1+m^{(i, i)}  +  Z^{(i)}_{i })^{-1} (1+m^{(j, j)}  +Z^{(j)}_{j })^{-1}   \\
&  \le  C |z|^2    | 1+  m|^{-2}   \left [  |Z^{(i)}_{i } | +  |Z^{(j)}_{j }|  +  (N \eta)^{-1} \right ]  \le (\log N)^{-  C}  |w|^{1/2}
\end{align}

Under the assumption $|w|$ is small enough and $\e^{1/4}|w|^{-1/2}\leq   |m| \leq (\log N) |w|^{-1/2}$,
the condition \eqref{e} is satisfied and, from \eqref{r1}, we have
\be  \label{r11}
|\cal Z_{i }|
\leq  \left [   \frac { 1} { N \eta} +  \frac { 1 } { (\log N)^C}  \right ] |w|^{-1/2}
\ee

Thus we have in this case  the estimate
\be\label{u1}
\left|G_{ii}-m \right|  \le N^{-1} \sum_j \left|G_{ii}-G_{jj}\right|\leq |G_{ii} G_{jj}|\left|G^{-1}_{ii}-G^{-1}_{jj}\right|  \le     |G_{ii} G_{jj}|   \left [   \frac { 1} { N \eta} +  \frac { 1 } { (\log N)^C}  \right ] |w|^{ 1/2}
\ee
Define the parameter
$$
\gamma (z, w) = \max_i |G_{ii}|  |w|^{1/2}
$$
Then \eqref{u1} and the assumption  $\e^{1/4}|w|^{-1/2}\leq   |m| \leq (\log N) |w|^{-1/2}$ imply that
$$
\gamma (z, w)  \le  C \log N + \frac { \gamma(z, w)^2  } { (\log N)^C}
$$
By continuity (in $\eta$)  method,  we have proved  $\gamma (z, w)  \le  C \log N$  (i.e., \eqref{53s}) assuming that  $\e^{1/4}|w|^{-1/2}\leq   |m| \leq (\log N) |w|^{-1/2}$   holds. Since $ |m|\leq (\log N) |w|^{-1/2}$ holds \hp{\zeta} by \eqref{52s},
to prove  \eqref{53s},  we only have to consider the last case $|m| \le \e^{1/4}|w|^{-1/2}$.

By \eqref{110c1} and \eqref{r2}, we have
$$
|\mG^{(i,\emptyset)}_{ii} | \ge |w|^{-1}  |1+m^{(i, i)}  + Z_{i}^{(i)}|^{-1} \ge \e^{-1/4} |w|^{-1/2}/2
 $$
Thus
$$
|m_\mG^{(i, \emptyset)}+   |z|^2 \mG_{ii}^{(i,\emptyset)} | \ge \e^{-1/4} |w|^{-1/2}/2 - \e^{ 1/4}|w|^{-1/2} - \frac 1 { N \eta}
\ge \e^{-1/4}|w|^{-1/2}/4
$$

By \eqref{r1} and the notation   $m_\mG^{(i, \emptyset)}+   |z|^2 \mG_{ii}^{(i,\emptyset)} = A_i + \ii B_i$, $A_i, B_i\in \R$, we have, for $\eta$ satisfies \eqref{eta}, that
 \begin{align} \label{58st}
| \cal  Z_{i }|
 \le   (\log N)^{-1}| B_i | +  \frac { (\log N)^{C} } { N \eta}
 \end{align}
  From \eqref{110},  we   have
\be\label{110-1}
\frac1{ |G_{ii}|  }
\ge      |w| \left  [| 1+  A_i + \ii B_i| - | \cal Z_{i }|  \right ]  \ge  |w| \left [  \frac { | 1+  A_i + \ii B_i|   } 2  -   (\log N)^{-C}    \right ]
\ge \e^{-1/4} |w|^{1/2}/8
\ee
We have thus proved  \eqref{53s}.

We now prove  \eqref{53.5s}.
From \eqref{110}, we have
\be\label{88}
 (wG _{ii})^{-1} + 1+  m_\mG^{(i, \emptyset )}
  =   -   [ |z|^2 \mG_{ii}^{(i, \emptyset)} +\cal Z_{i } ]
\ee

By the triangle inequality,  we have
\be\label{625}
| w z^2 G_{ii}    \mG_{ii}^{(i,\emptyset)}   |   \le   |w  G_{ii}   [|z|^2  \mG_{ii}^{(i,\emptyset)} +  \cal Z_i ]  | + | w  G_{ii} \cal Z_i|
\ee
Together with  \eqref{88} and using \eqref{r1} to bound  $ \cal Z_i$, with \eqref{eta}  we have
\begin{align}\label{626x}
| w z^2 G_{ii}    \mG_{ii}^{(i,\emptyset)}   |
 & \le
     | w  G_{ii}   [1+  m_\mG^{(i, \emptyset)} ] |  +  1
     +\varphi^{C_\zeta}|wG_{ii}| \sqrt{\frac{\im m_\mG^{(i, \emptyset)}+ |z|^2 \im\mG^{(i, \emptyset)}_{ii} }{N\eta}}
 \\\nonumber
 &  \le
   | w  G_{ii}   [1+  m_\mG^{(i, \emptyset)} ] |  +  1 +\frac {|m| | w  G_{ii}|} {(\log N)^C}   +\varphi^{-1}\sqrt{| w z^2 G_{ii}    \mG_{ii}^{(i,\emptyset)}   | }
 \end{align}
Then with  \eqref{52s} and \eqref{53s} on $m$ and $G_{ii}$, we obtain \eqref{53.5s}.

For any    eigenfunction with eigenvalues $ |\lambda_\al - E|   \le \eta$, we have
$$
 \frac{\eta |u_\al(i)|^2 }{(\lambda_\al-E)^2+\eta^2}  \le    |\im G_{ii}|\leq |G_{ii}|
$$
and thus
$$
 |u_\al(i)|^2    \le  2 \eta  |G_{ii}| \le \varphi^{C_\zeta} N^{-1},
$$
where we used (\ref{53s}) for $\eta=\varphi^{C_\zeta}N^{-1}|w|^{1/2}$.
Since this bound holds for all energy $E$, we have proved \eqref{52a}
(notice that the proof of \eqref{52a} for $w\geq \e$
follows by the same argument).

We now  prove \eqref{54s}. By the Cauchy-Schwarz inequality, we have
   \be\label{629}
   |G_{ij}|\leq \sum_\al\frac {|\lambda_\al-E| |u_\al(i)|^2 }{(\lambda_\al-E)^2+\eta^2}
   +
   \sum_\al\frac{|\lambda_\al-E||u_\al(j)|^2 }{(\lambda_\al-E)^2+\eta^2}  +\im G_{ii}+\im G_{jj}
   \ee
The imaginary parts, $\im G_{ii}+\im G_{jj}$,  can be bounded by \eqref{53s}.  The first term on the right hand side
is bounded by
\be\label{629a}
 \sum_\al\frac{|\lambda_\al-E||u_\al(i)|^2 }{(\lambda_\al-E)^2+\eta^2}\leq
  2\sum_{\al: \lambda_\al \geq E}
  \frac{(\lambda_\al-E)|u_\al(i)|^2 }{(\lambda_\al-E)^2+\eta^2}+|\re G_{ii}|
\ee
The real part, $\re G_{ii}$ is bounded again by \eqref{53s}.
Let $E_k=E+(2^k-1)\eta$ and $\eta_k=2^{k-1}\eta$. From \eqref{53s},  we have
\be \label{629b}
  \sum_{\al: \lambda_\al \geq E}
  \frac{(\lambda_\al-E)|u_\al(i)|^2 }{(\lambda_\al-E)^2+\eta^2}
  \leq\sum_{k=0}^{C\log N} \sum_{\al: E_k\leq  \lambda_\al \leq E_{k+1}}
  \frac{C\eta_k|u_\al(i)|^2 }{(\lambda_\al-E_k)^2+\eta_k^2}\leq C\log N|w|^{-1/2}
\ee
This concludes the proof of \eqref{54s} and Lemma \ref{a priori2}.
\end{proof}

\begin{lemma} \label{5.1} Suppose that $|w|+|z|=\OO(1)$.
For any $\zeta>0$, there exists $C_\zeta$  such that if \eqref{eta} holds then
we have
\be\label{55s}
  \left|\left[Y  G (z)\right]_{k l } \right |
 +\left|\left[ G  Y^*(z)\right]_{k l } \right |\leq  C(\log N)^2
\ee
and
\be\label{56s}
 \left|\left[Y  G  Y^*(z)\right]_{k l }-\delta _{kl}\right |
  \le C(\log N)^2|w|^{1/2}
\ee
and   \be\label{57s}
 \max_{ij}\left|[G^2]_{ij}\right|\leq  \max_{ij}\left|  |G|^2_{ij}\right| \le  \frac{C(\log N)^2}{\eta}  |w|^{-1/2}
\ee
and
\be\label{58ss}
 \max_{ij}\left|[YG^2]_{ij}\right|\leq \frac{C(\log N)^2}{\eta}
\ee
and
\be\label{59s}
\max_{ij}\left|[YG^2Y^*]_{ij}\right|\leq \frac{C(\log N)^2}{\eta} |w|^{1/2}
\ee
\hp{\zeta}.

 \end{lemma}

\begin{proof}
We first prove \eqref{55s} and \eqref{56s}.  From the general identity
$$
A (A^*  A - w)^{-1} A^*=1+ w ( A  A^*   - w)^{-1},
$$
we have
$$
(Y G Y^*)_{ij} = \delta_{ij} + w \mG_{ij}
$$
This proves \eqref{56s} with \eqref{53s}.   To prove \eqref{55s}, we notice first the identity
$$
Y G = \mG Y, \quad G Y^* = Y^* \mG.
$$
Recall the rank one perturbation formula
$$
(A +   \bv^* \bv)^{-1} = A^{-1} - \frac {  (A^{-1}  \bv^*)  (\bv A^{-1})} { 1 +   \bv A^{-1}   \bv^*}
$$ and Lemma \ref{lem: GmG}.  We have
$$
  \mG \by_i
= \frac { \mG^{ (i,\emptyset)} \by_i } { 1+ \langle \by_i ,\mG^{ (i,\emptyset)}  \by_i \rangle}.
$$
Together with the equation
$$
\left[G  _{ii} \right]^{-1}=-w-w\,\by_i^\dagger \mG^{(i, \emptyset )} \by_i
$$
(see Equation (6.18) in \cite{BouYauYin2012Bulk} for a derivation) and the definition of $\by_i$ as the $i$-th column of $Y $,  we have
\be \label{641}
 ( \mG  \by_i)_j
  =-wG_{ii} (\mG^{ (i,\emptyset)}\by_i )_j=-wzG_{ii}\mG_{ji}^{ (i,\emptyset)}    -  w  G_{ii} \sum_k \mG^{ (i,\emptyset)}_{jk} X_{ki}
  \ee
 From the large deviation lemma, we can bound
\be\label{642}
\left | \sum_k \mG^{ (i,\emptyset)}_{jk} X_{ki} \right |
 \le
\varphi^C \sqrt {N^{-1}  \sum_k |\mG^{ (i,\emptyset)}_{jk} |^2 }
 \le
 \varphi^C \sqrt{\frac{\im \mG_{ii}^{(i,\emptyset)}}{N\eta}}
  \le
  \im \mG_{ii}^{(i,\emptyset)}  + \frac {\varphi^C} { N \eta}
\ee
with high probability.
Hence
\be \label{643}
| ( \mG  Y )_{  ji} | = | ( \mG  \by_i)_j|
 \le \mathds{1}_{i\neq j}  | wzG_{ii}\mG_{ji}^{ (i,\emptyset)}|    + C (\log N)^4
  \ee
where we have used \eqref{53s},  \eqref{53.5s} and \eqref{eta}.  Now we estimate $\mG_{ji}^{ (i,\emptyset)}$, with \eqref{110}, \eqref{132}, \eqref{1328}, we have

\be\label{644}
\left| w G_{ii}\mG_{ji}^{ (i,\emptyset)}\right|
\leq
\varphi^C
\left|w^2 G_{ii}\mG_{ii}^{ (i,\emptyset)}\mG_{jj}^{ (i,i)}\right|
\sqrt{\frac{\im m_G^{(i,ij)}+|z|^2\im G_{jj}^{(i,ij)}}{N\eta}}
\ee
Then using $w  G_{ii}\mG_{ii}^{ (i,\emptyset)}\leq (\log N)^4$
 and
 $w\mG_{jj}^{ (i,i)}  G_{jj}^{(i,ij)}\leq (\log N)^4$ (see \eqref{53.5s}), we obtain it is less than $(\log N)^2$.  This proves \eqref{55s}.

By the Cauchy-Schwarz inequality, we have
$$
 |G|^2_{ij}\leq \sqrt{\sum_k|G_{ik}|^2}\sqrt{\sum_k|G_{jk}|^2}\leq \sqrt{\frac{\im G_{ii}}{\eta}}\sqrt{\frac{\im G_{jj}}{\eta}},
$$
 which implies \eqref{57s}.
Using $|w|\geq \im w =\eta$,   \eqref{54s} and \eqref{57s},  we have
$$
| (YG^2Y)_{ij} | = |\mG_{ij} + w \mG^2_{ij}|  \le  C |w|^{-1/2} + C |w|^{-1/2}\eta^{-1}  \le  C  |w|^{-1/2}\eta^{-1}
$$
This proves \eqref{59s}.   Notice that we claimed that all  estimates proved for $G$
are valid for $\mG$ as well. This is because that we can write $Y Y^*$ in $\mG$ as $A^* A$
where  $A = Y^*$. Hence all estimates hold for $G$ will hold for $\mG$.

By the the Cauchy-Schwarz inequality, we have
$$
|(YG^2)_{ij}| \le  |\sum_k (YG)_{ik} G_{kj} | \le | (YG G^* Y^*)_{ii}  | | (G^*G)_{ii}|.
$$
Together \eqref{57s} and \eqref{59s}, we have proved  \eqref{58ss}.
\end{proof}

\begin{remark}\label{rem:R}
The previous lemmas \ref{a priori2} and \ref{5.1} also hold if one entry of $X$ is supposed to vanish (this will be useful for us in the following of this
subsection).
Indeed, we just used the following facts: the independence of the entries,
the subexpotential decay, $\E X_{ab}=0$, and $\E |X_{ab}|^2=1/N$. We notice that only this last
condition is changed, and it was are only \eqref{110} and \eqref{110c}. Furthermore, for \eqref{110},
this affects the case $i=a$ as
$$
 G^{(\emptyset , T)} _{aa}
 =   - w^{-1}\left[1+  m_\mG^{(a, T)}-\frac1N \mG^{(a, T)}_{bb}+   |z|^2 \mG_{aa}^{(a,T)} +\cal Z^{(T)}_{a} \right]^{-1}
$$
The difference is therefore just $\frac1N \mG^{(a, T)}_{bb}$  which is of order $1/N$ of $m_\mG^{(a, T)}$, so the proof holds for
such matrices with a vanishing entry.
\end{remark}

  The following lemma is not useful for the proof of Theorem \ref{wfgb}, but it helps to understand how the Green's function comparison method works at here. Furthermore, some by-products are very useful for the whole proof of Theorem \ref{wfgb}.

\begin{lemma}\label{z11}
Suppose that  $|w|+|z|  <c$ for some fixed $c>0$, and that $\eta$ satisfies \eqref{eta}.
Assume that we have two ensembles $X_1, X_2$, both of them satisfying (\ref{subexp}), and with
matrix elements  moments matching up to order 3. Then we have
\be\label{544}
\left|   \E_{X^{(1)}}   (m(w,z))- \E_{X^{(2)}}   (m(w,z))\right| \leq \frac{ \varphi^C}{N\eta}.
\ee
\end{lemma}

\begin{proof}
For $k\in\llbracket 0,N^2\rrbracket$, define the following matrix $X_k$ interpolating between $X^{(1)}$ and $X^{(2)}$:
$$
X_k(i,j)=
\left\{
\begin{array}{ccc}
X^{(1)}(i,j)&{\rm\ if\ }&k<N(i-1)+j\\
X^{(2)}(i,j)&{\rm\ if\ }&k\geq N(i-1)+j\\
\end{array}
\right..
$$
Note that $X^{(1)}=X_0$ and $X^{(2)}=X_{N^2}$. A sufficient condition for \eqref{544} is that for any $k\geq 1$
   \be \label{nrpg}
\left|   \E_{X_k}   m(w,z)- \E_{X_{k-1}}  m(w,z)\right| \leq  \frac{ \varphi^C}{N^3\eta}
\ee
We are going to compare the Stieltjes transforms corresponding to $X_k$ and $X_{k-1}$ with a third one,
corresponding to the matrix $Q$ hereafter with deterministic $k$-th entry: noting $k=aN+b$
($a\in\llbracket 0,N-1\rrbracket$, $b\in\llbracket 1,N\rrbracket$) we define
the following $N\times N$ matrices (hereafter, $Y_\ell=X_\ell-z I$):
\begin{align}
v&=v_{ab}{\bf e}_{ab}=X^{(1)}(a,b){\bf e}_{ab},\label{def:vab}\\
u&=u_{ab}{\bf e}_{ab}=X^{(2)}(a,b){\bf e}_{ab},\label{def:uab}\\
Q&=Y_{k-1}-v=Y_{k}-u,\label{def:Q}\\
R&=(Q^* Q-wI)^{-1}\label{def:R}\\
\cal R&=(Q Q^*-wI)^{-1}\notag\\
S&=(Y^*_{k-1} Y_{k-1}-w I)^{-1}\label{def:S}\\
T&=(Y^*_{k} Y_{k}-w I)^{-1}\label{def:T}
 \end{align}
Then (\ref{nrpg}) holds if we can prove that
\begin{align}\label{p2}
\tr R-\E_{v_{ab}}\tr S &=  F + (\log N)^C \OO( N^{-2} \eta^{-1}  ),\\
\tr R-\E_{u_{ab}}\tr T&=  F + (\log N)^C \OO( N^{-2} \eta^{-1}  ),\notag
\end{align}
holds \hp{\zeta} where $\E_{v_{ab}}$ (resp. $\E_{u_{ab}}$) means an integration only with respect to the random variable $v_{ab}$
(resp. $u_{ab}$), and $F$ is a random variable, identical for both equations. We prove the first equation, the proof of
the other one being obviously the same. For this, we want to compare $S$ and $R$.
Denoting
$$
U=(1 +RQ^* v)^ {-1}R(1+ v^* QR)^{-1},
$$
we first compare $\tr U$ and $\tr R$, and then $\tr S$ and $\tr U$ ($S$ and $U$ are related by the simple formula (\ref{eqn:SU}) hereafter, with the notation (\ref{eqn:Omeg})).
For this first comparison, we introducing the notations
$$
p=(RQ^*)_{ba} v_{ab},\quad q= (QR)_{ab}v^*_{ba},
$$
appearing in the following identity obtained by expansion\footnote{All the expansions considered here converge with high probability. Anyways, they
aim at proving identities between rational functions, which just need to be checked for small values of the perturbation.}:
 \begin{align}
\tr U&=\tr(1+ v^* QR)^{-1} (1 +RQ^* v)^ {-1}R
= \sum_{k,l\geq 0}(-1)^{k+l}\tr\left((v^* QR)^k(RQ^* v)^l R\right)\notag\\
&=\tr R+ \sum_{k\geq 1}(-1)^{k}\tr\left((v^* QR)^k R\right)\notag
+\sum_{l\geq 1}(-1)^{l}\tr\left((RQ^* v)^l R\right)
+ \sum_{k,l\geq 1}(-1)^{k+l}\tr\left((v^* QR)^k(RQ^* v)^l R\right)\\
\tr U&=\tr R-  \frac{v^*_{ba} (QR^2)_{ab}}{1+q} - \frac{ ( R^2Q^*)_{ba}v _{ab}  }{1+p}+\frac{(QR^2Q^*)_{aa}|v_{ab}|^2R_{bb}}{(1+p)(1+q)}.\label{eqn:UR}
\end{align}
For this last equality, we extensively used that if $v_{ij}\neq 0$ then $(i,j)=(a,b)$.

To compare now $\tr S$ and $\tr U$, we introduce the notation
\begin{equation}\label{eqn:Omeg}
\Omega =
  v^* (Q R Q^* - 1)  v.
\end{equation}
A routine calculation yields
$  X^*_{k-1} X_{k-1}-w
   = U^{-1} - \Omega,
$
hence
\begin{equation}\label{eqn:SU}
S = (1 - U \Omega)^{-1} U.
\end{equation}
As  $Q RQ^*-I=w\cal R$, we have $\Omega_{ij}  = r \delta_{ib} \delta_{jb}$ where
$$
 r=|v_{ab} |^2 w\cal R_{aa}.
$$
Consequently, expanding (\ref{eqn:SU}) we get
$$
 S_{ij} =U_{ij}+  U_{ib} r U_{bj}
 + U_{ib} r U_{b b}rU_{bj}
 +U_{ib} r U_{b b}r U_{b b}rU_{bj}\cdots
$$
so after summation over $i=j$,
 \be\label{564}
 \tr S =\tr U+[U^2]_{bb}
 \, r(1-U_{bb}  r)^{-1}
 \ee
An argument similar to the one leading to (\ref{eqn:UR}) yields that, for any matrix $M$,
$$
[ (1 +RQ^* v)^ {-1}  M (1+ v^* QR)^{-1} ]_{bb} =  \frac{ M_{bb}}{(1+p)(1+q)},
$$
which yields, in our context,
$$
U_{bb}= \frac{ R_{bb}}{(1+p)(1+q)}, \
\frac{r}{1-U_{bb} r}=\frac{r(1+p)(1+q)}{(1+p)(1+q)-r  R_{bb}},\
[U^2]_{bb}=\frac{\left(R(1+ v^* QR)^{-1}(1 +RQ^* v)^ {-1}R\right)_{bb}}{{(1+p)(1+q)} }.
$$
The numerator of this last expression is also, by the same reasonning leading to (\ref{eqn:UR}),
$$
 (R^2)_{bb}-\frac{R_{bb}v^*_{ba}  (QR^2)_{ab} }{1+q}
 -\frac{(R^2 Q ^*)_{ba}v_{ab}R_{bb}  }{1+p}
 +\frac{(QR^2Q^*)_{aa}|v_{ab}|^2R_{bb}^2}{(1+p)(1+q)}
$$
Substituting the above expressions in (\ref{564}) and combining it with (\ref{eqn:UR}), we get
\begin{align}
\tr S-\tr R& =\frac{r (R^2)_{bb}}{(1+p)(1+q)-r  R_{bb}}\nonumber
 \\
 &+ \frac{(1+p)(1+q)}{(1+p)(1+q)-r  R_{bb}}
\left( - \frac{v^*_{ba} (QR^2)_{ab}}{1+q}- \frac{ ( R^2Q^*)_{ba}v _{ab}  }{1+p}+\frac{(QR^2Q^*)_{aa}|v_{ab}|^2R_{bb}}{(1+p)(1+q)}\right) \label{pos571}
\end{align}

In the above formula, $v_{ab}$ only appears through  $p, q$ (in a linear way)  and $r$ (in a
quadratic way). All other terms can be bounded thanks to
the  estimates for   $S$ from  Lemmas \ref{a priori2} and \ref{5.1}, which also holds for $R$, as stated in Remark \ref{rem:R}.
More precisely, the following bounds hold \hp{\zeta}, for any choice of $a$ and $b$ (including possibly $a=b$), and
under the assumption (\ref{eta}):
\begin{equation}\label{eqn:bounds}
\begin{array}{rll}
 |(RQ^*)_{ba}|,\ |(QR)_{ab}| &\le (\log N)^C &\ {\rm by\ (\ref{55s}),}\\
 |R_{bb}|,\ |\mathcal{R}_{aa}|&\leq   (\log N)^C  |w|^{-1/2}&\ {\rm by\  (\ref{53s}),}\\
|( R^2 )_{bb}|& \le (\log N)^C|w|^{-1/2}\eta^{-1} &\ {\rm by\  (\ref{57s}),}\\
|(Q R^2)_{ab} | &\le  (\log N)^C\eta^{-1}&\ {\rm by\  (\ref{58ss}),}\\
|(Q R^2 Q^*)_{aa}|&\le (\log N)^C |w|^{1/2} \eta^{-1}&\ {\rm by\  (\ref{59s}).}\\
\end{array}
\end{equation}

 Therefore, if we make an expansion  of \eqref{pos571} with respect to
$v_{ab}$, we get for example for the first term of the sum in \eqref{pos571} \hp{\zeta} (remember that $v_{ab}$ satisfies the subexponential decay property)
$$
\frac{r (R^2)_{bb}}{(1+p)(1+q)-r  R_{bb}}
=
f + \OO\left(  w (R^2)_{bb}\mathcal{R}_{aa}|v_{ab}|^4\max\left((RQ^*)_{ba}^2,(QR)_{ab}^2,w R_{bb}\mathcal{R}_{aa}\right) \right)\\
=f + \OO\left(\frac{ \varphi^{C_\zeta}}{N^2\eta}\right)
$$
where $f$ is a polynomial of degree 3 in $v_{ab}$.
A calculation shows that the same estimate, of type $g + \OO\left(\frac{ \varphi^{C_\zeta}}{N^2\eta}\right)$
with $g$ of degree 3 in $v_{ab}$, holds when expanding
the second term of the sum $\eqref{pos571}$.
This finishes the proof of (\ref{p2}).
\end{proof}

\begin{proof}[Proof of Lemma \ref{wfgb}]
First since Theorem \ref{z1} holds in the Gaussian case (cf.  Theorem \ref{z1Ginibre}),
then with  \eqref{nkwT},  the estimate on the smallest eigenvalues (Lemma \ref{smallest}) and the  largest eigenvalue (Lemma \ref{largest}),  the estimate
\eqref{mjs} holds in the case of centered and reduced complex Gaussian entries. Furthermore, using  \eqref{yxl},
we deduce that Lemma \ref{wfgb} holds in the Gaussian case.
We will therefore use the same method as in the proof of Lemma \ref{z11},
replacing the matrix elements one by one to extrapolate from the Ginibre ensemble to the general setting.
We assume that $X^{(1)}=X_0$ is the Gaussian case and  $X^{(2)}=X_{N^2}$ is the ensemble for which we want to prove
Lemma \ref{wfgb}.
We know that
$
Z^{(f)}_{X^{(1)},{\rm c}} \prec N^{C\e}
$, and we even know that, for any { fixed }$p>0$, $\E(|Z^{(f)}_{X^{(1)},{\rm c}}|^p)\leq N^{C\e p}$
(Theorem \ref{z1Ginibre} is proved by bounding the moments).
We will prove that, for { any fixed }$p\in2\N$ and $N$ sufficiently large,
\begin{equation}\label{eqn:induc}
\left|\E \left((Z^{(f)}_{X_{N^2},{\rm c}})^p\right)\right| \leq { C_p}N^{\tilde C\e p},
\end{equation}
for some $\tilde C>C$  independent of $N$ and $p$.
This is sufficient for our purpose, by the Markov inequality.

From $X_k$ and $X_{k-1}$, we defined the matrices $R$, $S$, $T$, in (\ref{def:R}), (\ref{def:S}), (\ref{def:T}),
From equation (\ref{pos571}), one can write
$$
\tr S-\tr R=P_{w,z,Q}(v_{ab})+B_{w,z,S}v_{ab}^4,
$$
where $P_{w,z,Q}$ is a degree 3 polynomial whose coefficients depend only on $Q$.
Moreover, a careful analysis of (\ref{pos571}), using the estimates (\ref{eqn:bounds}), proves that the coefficients of $P$ are
$\OO_\prec(\eta^{-1})$, and that $B$ is also $\OO_\prec(\eta^{-1})$.
Note that $v_{ab}^4=\OO_\prec(N^{-2})$, hence we obtain (hereafter, $\tilde Q=X_{k-1}-v=X_{k}-u$)
\begin{align*}
Z^{(f)}_{X_{k-1},{\rm c}}-Z^{(f)}_{\tilde Q,{\rm c}}=&
\int\Delta f(\xi)\int_I\chi(\eta)\phi'(E)\OO_{\prec}(\eta^{-1}N^{-2})\rd E\rd \eta\rd\xi\rd\overline\xi\\
&+
\int\Delta f(\xi)\int_I\chi(\eta)\phi'(E)\re\left(P_{w,z,Q}(v_{ab})
-P_{w,z,Q}(0)\right)\rd E\rd \eta\rd\xi\rd\overline\xi
\end{align*}
Noting that $|\phi'(E)|\leq (1+\log E)E^{-1}\1_{ 4\lambda_+>E>N^{\e-2}}$,
the first term in the above sum is
$$
\OO_\prec\left(N^{-2}\int|\Delta f(\xi)|\int_{I\cap\{ 4\lambda_+>E>N^{\e-2}\}}\eta^{-1}E^{-1}\rd E\rd \eta\rd\xi\rd\overline\xi\right)
=
\OO_\prec\left(N^{-2}|\Delta f|_{L^1}\right)=\OO_\prec\left(N^{-2}\right),
$$
where we omit the dependence in $f$ in the previous and next estimates, as $f$ does not depend on $N$.
Concerning the second term, it is of type
$
\re\mathcal{P}_{f,Q}(v_{ab}),
$
where  $\mathcal{P}$ has degree 3, vanishes at $0$, with coefficients of order $\OO_\prec(1)$ being independent of $v_{ab}$ and $u_{ab}$.
We therefore have
\begin{align}\label{eqn:rec}
Z^{(f)}_{X_{k-1},{\rm c}}&=Z^{(f)}_{\tilde Q,{\rm c}}+\Delta_{k-1},\
\Delta_{k-1}= \mathcal{P}_{f,Q}(v_{ab})+\OO_\prec(N^{-2}),\\
\notag
Z^{(f)}_{X_{k},{\rm c}}&=Z^{(f)}_{\tilde Q,{\rm c}}+\Delta_k,\
\Delta_k= \mathcal{P}_{f,Q}(u_{ab})+\OO_\prec(N^{-2})
\end{align}
 We can decompose
$$
(Z^{(f)}_{X_{k-1},{\rm c}})^p-(Z^{(f)}_{X_{k},{\rm c}})^p=\sum_{j=0}^{p-1} \binom{p}{j}
(Z^{(f)}_{\tilde Q,{\rm c}})^j (\Delta_{k-1}^{p-j}-\Delta_k^{p-j}).
$$
Since $\cal P$ has no constant term
and the first three moments of $v_{ab}$ and $u_{ab}$ coincide, we get
\begin{equation}\label{eqn:rec2}
\left | \E\left((Z^{(f)}_{X_{k-1},{\rm c}})^p\right)-\E\left((Z^{(f)}_{X_{k},{\rm c}})^p\right) \right | \le
\sum_{j=0}^{p-1}
\E\left|(Z^{(f)}_{\tilde Q,{\rm c}})^j\right| \OO_\prec(N^{-2})
\le
N^{-2}\left( \OO_\prec(1)+\E\left((Z^{(f)}_{\tilde Q,{\rm c}})^p\right)\right),
\end{equation}
where we used that $p$ is even,     $\E\left|(Z^{(f)}_{\tilde Q,{\rm c}})^j\right|\leq \E\left|(Z^{(f)}_{\tilde Q,{\rm c}})^p\right|^{j/p}$ and
\be
 \E\left|(Z^{(f)}_{\tilde Q,{\rm c}})^p\right|^{j/p}N^{\delta}\leq  \E\left|(Z^{(f)}_{\tilde Q,{\rm c}})^p\right|+N^{\delta\frac{p}{p-j}},
\ee
by Young's inequality (remember that $j\in\llbracket 0, p-1\rrbracket$)
Moreover, from (\ref{eqn:rec}),
\begin{align*}
\E\left((Z^{(f)}_{\tilde Q,{\rm c}})^p\right)&\leq \E\left((Z^{(f)}_{X_{k-1},{\rm c}})^p\right)+
\sum_{j=1}^p\binom{j}{p}\E\left(\left|Z^{(f)}_{X_{k-1},{\rm c}}\right|^{p-j}|\Delta_{k-1}|^j\right)\\
&\leq
\E\left((Z^{(f)}_{X_{k-1},{\rm c}})^p\right)+
\sum_{j=1}^p\binom{j}{p}\E\left(\left|Z^{(f)}_{X_{k-1},{\rm c}}\right|^{p}\right)^{\frac{p-j}{p}}
\E\left(\left|\Delta_{k-1}\right|^{p}\right)^{\frac{j}{p}}, \\
\E\left((Z^{(f)}_{\tilde Q,{\rm c}})^p\right)&\leq \left(\OO_\prec(1)+\E\left((Z^{(f)}_{X_{k-1},{\rm c}})^p\right)\right),
\end{align*}
where we used H\"older's inequality and the trivial bound $\Delta_{k-1}=\OO_\prec(1)$ in the last equation.
Using the last bound and  (\ref{eqn:rec2}), we obtain for any $\delta>0$
$$
\left | \E\left((Z^{(f)}_{X_{k-1},{\rm c}})^p\right)-\E\left((Z^{(f)}_{X_{k},{\rm c}})^p\right) \right |
\le
N^{-2}\left(N^{\delta}+\E\left((Z^{(f)}_{X_{k-1},{\rm c}})^p\right)\right).
$$
It implies
\be\label{36555}
\E\left((Z^{(f)}_{X_{k },{\rm c}})^p\right)+ N^{\delta}\leq (1+N^{-2})\left(\E\left((Z^{(f)}_{X_{k-1 },{\rm c}})^p\right)+N^{\delta}\right)
\ee
As $\E\left((Z^{(f)}_{X_{0},{\rm c}})^p\right)\leq N^{C\e p}$
and we
obtain (\ref{eqn:induc}) by iterating \eqref{36555}. It completes the proof of Lemma \ref{wfgb}.
\end{proof}

\begin{proof}[Proof of Lemma \ref{wfgbT}]
Notice that in the previous proof,  the third moment condition was first  used in \eqref{eqn:rec2}.
So  all  equations up to and including  \eqref{eqn:rec} are still valid.
Introducing the notation  $Y=Z/\sqrt{N}$, we have, instead of \eqref{eqn:rec2}, the following bound
\begin{equation}\label{eqn:rec2T}
\left | \E\left((Y^{(f)}_{X_{k-1},{\rm c}})^p\right)-\E\left((Y^{(f)}_{X_{k},{\rm c}})^p\right)  \right | \le
\sum_{j=0}^{p-1}
\E\left|(Y^{(f)}_{\tilde Q,{\rm c}})^j\right| \OO_\prec(N^{-2})
\le
 N^{-2}\left( \OO_\prec(1)+\E\left((Y^{(f)}_{\tilde Q,{\rm c}})^p\right)\right).
\end{equation}
Following the rest of the argument in the proof of Lemma \ref{wfgb},   we have proved that
\be
Y_{X,{\rm c}}^{(f)}\prec N^{ C\e}c_f,
 \ee
 and this completes  the proof of Lemma \ref{wfgbT}.
\end{proof}

 \section{The local circular law for the Ginibre ensemble}\label{App:Ginibre}

In this section, we derive the local circular law on the edge for the Ginibre ensemble (\ref{eqn:Ginibre}). This is required in the proof of Lemma
\ref{wfgb}, which proceeds by comparison with the Gaussian case.
Along the proof, we will need the following partition of $\C$, with distinct asymptotics of the correlation function $K_N$ for each domain (in the following we will always take $\e_N=\frac{\log N}{\sqrt{N}}$):
\begin{align*}
\Omega_1&=\{|z|<1-\e_N\},\ {\rm the\ bulk},\\
\Omega_2&=\{|z|>1+\e_N\},\\
\Omega_3&=\C-(\Omega_1\cup\Omega_2),\ {\rm the\ edge}.
\end{align*}

We first consider the case when the
test function is supported in $\Omega_1$ (this step is required even for the circular law on the edge, as the
support of the test functions overlaps the bulk), then when it can overlap $\Omega_3$.
The first case is directly adapted from
\cite{AmeHedMak2011}, the second one requires some more work.

In the following, we note $e_N(z)=\sum_{\ell=0}^{N}\frac{z^\ell}{\ell!}$ for the partial sums of the exponential function.

\subsection{The bulk case. } We prove the following local cirular law for the Ginibre ensemble,
with some more precision: the local convergence towards the Gaussian free field holds in the bulk,
generalizing the global convergence result obtained in \cite{RidVir2007}.

\begin{theorem}\label{lclGinibre}
Let $0<a<1/2$, $f$ be a smooth non-negative function with compact support,
and $z_0$ be such that $f_{z_0}$ is supported in $\Omega_1$.
Let $\mu_1,\dots,\mu_N$ be distributed as the eigenvalues of a Ginibre random matrix (\ref{eqn:Ginibre}).
Then the local circular law (\ref{yjgq}) holds.

More generally,
if $f $ additionally may depend on $N$,
such that   $\|f\|_\infty\leq C$, $\|f'\|_\infty\leq N^C$,
the local circular law holds when at distance at least $\e_N$ from the unit circle in the following sense.
Define $f^{(N)}(z)=f(N^a(z-z_0))$, for $z_0$ (depending on $N$) such that $f^{(N)}$ is supported on $\Omega_1$.
Denote $\cum_N(\ell)$ the
$\ell$-th cumulant of $X^{(N)}_f=\sum f^{(N)}(\mu_j)-N^{1-2a}\frac{1}{\pi}
\int_\C f(z)\rd z\rd \bar z$, and denote\footnote{$\!|\!|\!| f\!|\!|\!|_\infty:=
\max(\|f'\|_\infty^3,\|f'\|_\infty\|f''\|_\infty,\|f^{(3)}\|_\infty)$.}
$\delta_N=N^{3a-\frac{1}{2}}\Vol(\supp(\nabla f^{(N)}))\!|\!|\!| f\!|\!|\!|_\infty$.
Then as $N\to\infty$, for $\ell\geq 1$,
\begin{equation}\label{eqn:cumulants1}
\cum_N(\ell)=
\left\{
\begin{array}{ll}
\frac{1}{4\pi}\int|\nabla f(z)|^2\rdA(z)+\OO\left(\delta_N\right)&{\ \mbox{if}\ }\ell=2,\\
\OO\left(\delta_N\right)&{\ \mbox{if}\ }\ell\neq2.
\end{array}
\right.
\end{equation}
In particular, noting $\sigma^2=\frac{1}{4\pi}\int|\nabla f(z)|^2\rdA(z)$, if $\delta_N\ll \sigma^2$, the linear statistics  $\frac{1}{\sigma}X^{(N)}_f$ converge in law to a reduced centered Gaussian
random variable.
\end{theorem}

As a first step in the proof, we need the following elementary estimate.

\begin{lemma}\label{lem:KnBulk}
Let
\begin{equation}\label{eqn:kN}
k_N(z_1,z_2)=\frac{N}{\pi}e^{-\frac{N}{2}(|z_1|^2+|z_2|^2-2z_1\overline{z_2})}.
\end{equation}
There is some $c>0$ such that uniformly in $|z_1\overline{z_2}|<1-\e_N$ we have
$$
K_N(z_1,z_2)=k_N(z_1,z_2)+\OO(e^{-c (\log N)^2}),
$$
where $K_N$ is the Ginibre kernel (\ref{eqn:GinibreKernel}).
\end{lemma}

\begin{proof}
This is elementary from the following calculation, for any $|z_1z_2|<1$:
\begin{align}
K_N(z_1,z_2)&=\frac{N}{\pi}e^{-\frac{N}{2}(|z_1|^2+|z_2|^2}\notag
\left(e^{Nz_1\overline{z_2}}-\sum_{\ell\geq N}\frac{(Nz_1\overline{z_2})^\ell}{\ell!}\right)\\
&=k_N(z_1,z_2)+\OO\left(Ne^{-\frac{N}{2}(|z_1|^2+|z_2|^2}\frac{|Nz_1\overline{z_2}|^N}{N!}\sum_{m\geq 0}|z_1,\overline{z_2}|^m\right)\notag\\
&=k_N(z_1,z_2)+\OO\left(Ne^{-N((|z_1|^2+|z_2|^2)/2-1-\log|z_1z_2|)}\frac{1}{1-|z_1z_2|}\right)\notag\\
K_N(z_1,z_2)&=k_N(z_1,z_2)+\OO\left(Ne^{-\frac{N}{2}(1-|z_1z_2|)^2}\frac{1}{1-|z_1z_2|}\right)\label{eqn:KNInside}
\end{align}
where we used Stirling's formula and the fact that $(|z_1|^2+|z_2|^2)/2-1-\log|z_1z_2|>
|z_1z_2|-1-\log|z_1z_2|>\frac{1}{2}(1-|z_1z_2|)^2$.
\end{proof}

The next following estimate about $K_N$ will be used to bound it for any distant $z_1$ and $z_2$.

\begin{lemma}\label{lem:fastDecrease}
There is a constant $c>0$ such that for any $N\in\N^*$ and any $z_1, z_2$ in $\C$,
$$
\left|K_N(z_1,z_2)\right|\leq
c N\left(e^{-N\frac{|z_1-z_2|^2}{2}}+\frac{|z_1\overline{z_2}|^{N+1}}
{1+\sqrt{N}|1-z_1\overline{z_2}|}\ e^{-N\left(\frac{|z_1|^2+|z_2|^2}{2}-1\right)}\right).
$$
\end{lemma}

\begin{proof}
In the case $|z_1\overline{z_2}-1|<1/\sqrt{N}$, by case $(i)$ in Lemma \ref{lem:close1},
we get
$$
K_N(z_1,z_2)=\frac{N}{\pi}e^{-\frac{|z_1|^2+|z_2|^2}{2}}e^{Nz_1\overline{z_2}}\left(\frac{1}{2}\erfc(\sqrt{N}\mu(z_1\overline{z_2}))+\OO(N^{-1/2})\right).
$$
As $\erfc$ is uniformly bounded in $\C$, we get $|K_N(z_1,z_2)|=\OO(e^{-N\frac{|z_1-z_2|^2}{2}})$.

In the case $\frac{1}{\sqrt{N}}\leq|z_1\overline{z_2}-1|<\delta$, for $\delta$ fixed and small enough, the estimate was obtained
in \cite{AmeOrt2011} (cf. the proof of Lemma 8.10 there, based on Lemma \ref{lem:close1} here).
Finally, in the case $|z_1\overline{z_2}-1|\geq\delta$, an elementary calculation from
Lemma \ref{lem:far1} implies that the result holds, no matter that $z_1\overline{z_2}$
is in $D-U$ or $D^c-U$.
\end{proof}

\begin{proof}[Proof of Theorem \ref{lclGinibre}]
Proving equation  (\ref{eqn:cumulants1}) is sufficient: this convergence of the cumulants is well-known to imply the weak convergence to the Gaussian distribution, and it also yields the local circular law (\ref{yjgq}):
if $m_N(\ell)$ denotes the $\ell$-th moment of
$X^{(N)}_f$, one can write
$$
m_N(\ell)=\sum_{i_1+\dots+i_\ell=\ell}\lambda_{i_1,\dots,i_k}\cum_N(i_1)\dots \cum_N(i_\ell)
$$
for some universal constants $\lambda_{i_1,\dots,i_\ell}$'s, where the $i_j$'s are in $\N^*$.
Hence writing $\Phi=\left(\sigma^2+
\delta_N\right)^{1/2}$
we get
$
m_N(\ell)=\OO\left(\Phi^{k}\right)
$. Consequently, for any $\e>0$,
$$
\Prob(|X^{(N)}_f|>N^{\e}\Phi)\leq m_N(\ell) N^{-\ell\e}\Phi^{-k}=\OO(N^{-\ell\e}),
$$
which concludes the proof of the local circular law by choosing $\ell$ large enough.

To prove (\ref{eqn:cumulants1}), first note that
from \ref{eqn:kN} evaluated on the diagonal $z_1=z_2$,
$\cum_N(1)=\OO(e^{-c(\log N)^2})$, so we consider now cumulants of order $\ell\geq 2$.
Due to a peculiar integrability structure of determinantal point processes,
the cumulants have an easy form for the Gaussian matrix ensembles, as observed first by Costin and Lebowitz \cite{CosLeb1995},
and used later by Soshnikov \cite{Sos2000}. For our purpose, in the context of the Ginibre ensemble,
we will use the following very useful expression (\ref{cumAHM}) due to Ameur, Hedenmalm and Makarov \cite{AmeHedMak2011}, which requires first the following notations. The usual differential operators are noted
$\partial=\frac{1}{2}(\partial_x-\ii\partial_y),
\bar\partial=\frac{1}{2}(\partial_x+\ii\partial_y),
\Delta_\ell=\partial_1\bar\partial_1+\dots+\partial_\ell\bar\partial_\ell$, we will also make use of the
length-$\ell$ vectors $z\1_\ell=(z,\dots,z)$,
$h=(h_1,\dots,h_\ell)$, and the following function often appearing in the combinatorics of cumulants for determinantal point processes:
$$F_\ell(z_1,\dots,z_\ell)=\sum_{j=1}^\ell\frac{(-1)^{j-1}}{j}\sum_{k_1+\dots+k_j=\ell,k_1,\dots,k_j\geq 1}
\frac{\ell!}{k_1!\dots k_j!}\prod_{m=1}^j f^{(N)}(z_m)^{k_m}.
$$
Then the $\ell$-th cumulant of $X_f^{(N)}$ is then
\begin{equation}\label{eqn:Fl}
\cum_N(\ell)=\int_{\C^{\ell}} F_\ell(z_1,\dots,z_\ell)K_N(z_1,z_2)K_N(z_2,z_3)\dots K_N(z_\ell,z_1)\A(z_1,\dots,z_\ell),
\end{equation}

Remarkably, some combinatorics can prove that $F_\ell(z\1_\ell)=0$:
it vanishes on the diagonal.
Moreover, the following  approximations of $F_\ell$ close to the diagonal will be useful:
$$
F_\ell(z\1_\ell+h)=\sum_{j\geq 1} T_j(z,h),\ {\rm where}\
T_j(z,h)=\sum_{|\alpha+\beta|=j}(\partial^\alpha\bar\partial^\beta F_\ell)(z\1_\ell)\frac{h^\alpha\bar h^\beta}{\alpha!\beta!}.$$
We will also need the third error term after second order approximation, $r(\lambda,h)=F_\ell(z\1_\ell+h)-T_1(z,h)-T_2(z,h)$.
The last notation we will need from \cite{AmeHedMak2011} is
$Z_\ell(z)=\sum_{i<j}\left(\partial_i\bar\partial_j F_\ell\right)(z\1_\ell)$, and
the area measure in $\C^\ell$ will be noted
$\A(z,h_1,\dots,h_\ell)=\A(z)\A(h_1)\dots\A(h_\ell)$, where $\A(z)=\rd^2z$.
Then, as proved in \cite{AmeHedMak2011}, for $\ell\geq 2$
\begin{equation}\label{cumAHM}
\cum_N(\ell)=A_N(\ell)+B_N(\ell)+C_N(\ell)+D_N(\ell)+E_N(\ell)
\end{equation}
where
\begin{align*}
A_N(\ell)&=\int_{\C^3}\Re\left(\sum_{i\neq j}(\partial_i\partial_j F_\ell)(z\1_\ell) h_1h_2\right) K_N(z,z+h_1)K_N(z+h_1,z+h_2)K_N(z+h_2,z)\A(z,h_1,h_2)\\
B_N(\ell)&=\Re \int_{\C^2}  \sum_{i}({\partial_i}^2 F_\ell)(z\1_\ell) h_1^2 K_N(z,z+h_1)K_N(z+h_1,z)\A(z,h_1)  \\
C_N(\ell)&=2\int_{\C^3}\Re\left(  Z_\ell(z) h_1\bar h_2\right) K_N(z,z+h_1)K_N(z+h_1,z+h_2)K_N(z+h_2,z)\A(z,h_1,h_2)\\
D_N(\ell)&=\int_{\C^2}(\Delta_\ell F_\ell)(z\1_\ell)|h_1|^2 K_N(z,z+h_1)K_N(z+h_1,z)\A(z,h_1)\\
E_N(\ell)&=\int_{\C^{\ell+1}}r(z,h) K_N(z,z+h_1)K_N(z+h_1,z+h_2)\dots K_N(z+h_\ell,z)\A(z,h_1,\dots,h_\ell).
\end{align*}
Remember that
the support of $f^{(N)}$ is at distance $\e_N$ from the unit circle, so each one of the above integrands vanishes if $z$ is out of
the disk with radius $1-\e_N$. Therefore, by using Lemma \ref{lem:fastDecrease} to restrict the domain, and then Lemma
\ref{lem:KnBulk} to approximate the kernel strictly inside the unit disk, one easily gets that
$$
\cum_N(\ell)=\tilde A_N(\ell)+\tilde B_N(\ell)+\tilde C_N(\ell)+\tilde D_N(\ell)+\tilde E_N(\ell)+\OO(e^{-c (\log N)^2}),
$$
for some $c>0$, where
\begin{align*}
\tilde A_N(\ell)&=\int_{\C^3\cap\{\|h\|_\infty<\e_N\}}\Re\left(\sum_{i\neq j}(\partial_i\partial_j F_\ell)(z\1_\ell) h_1h_2\right) k_N(z,z+h_1)k_N(z+h_1,z+h_2)k_N(z+h_2,z)\A(z,h_1,h_2)\\
\tilde B_N(\ell)&=\Re \int_{\C^2\cap\{\|h\|_\infty<\e_N\}}  \sum_{i}({\partial_i}^2 F_\ell)(z\1_\ell) h_1^2 k_N(z,z+h_1)k_N(z+h_1,z)\A(z,h_1)  \\
\tilde C_N(\ell)&=2\int_{\C^3\cap\{\|h\|_\infty<\e_N\}}\Re\left(  Z_\ell(z) h_1\bar h_2\right) k_N(z,z+h_1)k_N(z+h_1,z+h_2)k_N(z+h_2,z)\A(z,h_1,h_2)\\
\tilde D_N(\ell)&=\int_{\C^2\cap\{\|h\|_\infty<\e_N\}}(\Delta_\ell F_\ell)(z\1_\ell)|h_1|^2 k_N(z,z+h_1)k_N(z+h_1,z)\A(z,h_1)\\
\tilde E_N(\ell)&=\int_{\C^{\ell+1\cap\{\|h\|_\infty<\e_N\}}}r(z,h) k_N(z,z+h_1)k_N(z+h_1,z+h_2)\dots k_N(z+h_\ell,z)\A(z,h_1,\dots,h_\ell).
\end{align*}	
Following closely the proof  in \cite{AmeHedMak2011},  integrating the $h_i$'s
for fixed $z$,  the terms $\tilde A_N(\ell), \tilde B_N(\ell), \tilde C_N(\ell)$ are $\OO(e^{-c (\log N)^2})$.
Moreover, \cite{AmeHedMak2011} proved that, for $\ell\geq 3$, $\Delta_\ell F_\ell$ vanishes on the diagonal,
and that $\Delta_2F_2(z_1,z_2)\mid_{z_1=z_2=z}=\frac{1}{2}|\nabla f^{(N)}(z)|^2$. Consequently,
$\tilde D_N(\ell)=0$ if $\ell\geq 3$ and
\begin{multline*}
\tilde D_N(2)=\frac{1}{2}\int_{\C^2}|\nabla f^{(N)}(z)|^2|h|^2\frac{N^2}{\pi^2}e^{-N|h|^2}\A(z,h)+\OO(e^{-c (\log N)^2})\\=\frac{1}{4\pi}\int_\C|\nabla f^{(N)}(z)|^2\A(z)+\OO(e^{-c (\log N)^2})
=
\frac{1}{4\pi}\|\nabla f\|_2^2+\OO(e^{-c (\log N)^2}).
\end{multline*}
To finish the proof, we need to bound the error term $\tilde E_N(\ell)$.
We divide the set of possible $(z,h)$ in two parts. Following \cite{AmeHedMak2011}, define
$
Y_{n,\ell}=\{(z,h)\in\C^{\ell+1}: z\in\supp(\nabla f^{(N)})\}
$.
On the complement of $Y_{n,k}$,
 $z\not\in\supp(\nabla f^{(N)})$, and then $r(z,h)=0$.
On $Y_{n,\ell}$, $r(z,h)$ is of order $\OO(\|{F_\ell}^{(3)}\|_\infty\|h\|_\infty^3)
=
\OO(N^{3a}  \!|\!|\!| f\!|\!|\!|_\infty\e_N^3)
$,
each $k_N$ is $\OO(N)$, and
the domain of integration has size
$\OO\left(\Vol(\supp(\nabla f^{(N)}))\e_N^{2\ell}\right)
$,
so
$$
\tilde E_N(\ell)=\OO\left((\log N)^C N^{3a-\frac{1}{2}} \!|\!|\!| f\!|\!|\!|_\infty
\Vol(\supp(\nabla f^{(N)}))
\right),
$$
concluding the proof.
\end{proof}

\subsection{The edge case. }

In the edge case, i.e. $|z_0|\in\Omega_3$, we have the following theorem for the Ginibre ensemble.

\begin{theorem}\label{z1Ginibre}
Let $\mu_1,\dots,\mu_N$ be distributed as the eigenvalues of a Ginibre random matrix (\ref{eqn:Ginibre}),
and $z_0\in \Omega_3$.  Suppose that  $f$ is smooth and compactly supported, and let $f_{z_0}(z)=N^{2a}f(N^{a}(z-z_0))$.
Then for any $0<a<1/2$, the estimate (\ref{yjgq2}) holds.
\end{theorem}

We begin with the proper bound on the first cumulant, noting as previously $f^{(N)}(z)=f(N^a(z-z_0))$,
and $X_f^{(N)}=\sum f(N^a(\mu_j-z_0))-\frac{N}{\pi}\int_D f^{(N)}(z)\A(z)$.

\begin{lemma}\label{lem:Expectation}
With the previous notations, for any $\e>0$, $\E(X_f^{(N)})=\OO(N^{-\frac{1}{2}+\e})$ as $N\to\infty$.
\end{lemma}

\begin{proof}

From the definition of the 1-point correlation function,
$$
\E(X_f^{(N)})=N
\int_\C f^{(N)}(z)\left(\frac{1}{N}K_N(z,\bar z)-\frac{\chi_D(z) }{\pi}\right)\rd z\rd\bar z.
$$
From Lemma \ref{lem:KnBulk},  as $f$ is bounded,
\begin{multline*}
N\int_{\Omega_1}
f^{(N)}(z)\left(\frac{1}{N}K_N(z,\bar z)-\frac{1}{\pi}\right)\rd z\rd\bar z\\
=
N\int_{\Omega_1}f^{(N)}(z)\left(\frac{1}{N}k_N(z,\bar z)-\frac{1}{\pi}\right)\rd z\rd\bar z
+\OO\left(e^{-c(\log N)^2}\right)=\OO\left(e^{-c(\log N)^2}\right)
\end{multline*}
where $k_N$ is defined in (\ref{eqn:kN}).
On $\Omega_2$, by Lemma \ref{lem:fastDecrease}, $K_N(z,\bar z)=\OO(e^{-c(\log N)^2})$.  Consequently,
$$\int_{\Omega_2} f^{(N)}(z)K_N(z,\bar z)\rd z\rd\bar z=\OO\left(
e^{-c(\log N)^2}\right).$$
Without loss of generality, we consider the case $|z_0|=1$ for simplicity of notations.
On $\Omega_3$, we now use Lemma \ref{lem:close1Real}
which yields that, uniformly on $-\e_N<\e<\e_N$,
\begin{equation}\label{eqn:polar}
K_N(1+\e,1+\e)=\frac{N}{\pi}\left(\1_{\e<0}+\frac{1}{\sqrt{2}}\frac{\mu(t)t}{t-1}\erfc(\sqrt{N-1}\mu(t))\left(1+\OO\left(\frac{1}{\sqrt{N}}\right)\right)\right),
\end{equation}
where $t=\frac{N}{N-1}(1+\e)^2$.
Take any smooth function of type $g(1+\e)=G(N^a\e)$. By polar integration, we just need to prove that
the following quantity is $\OO(N^{a-\frac{1}{2}+\e})$:
\begin{align}
&\int_{-\e_N}^{\e_N}(1+\e)g(1+\e)\left(K_N(1+\e,1+\e)-\frac{N}{2\pi}\right)\rd\e\notag\\
=&\int_{-\e_N}^{\e_N}(1+\e)(g(1+\e)-g(1))\left(K_N(1+\e,1+\e)-\frac{N}{2\pi}\right)\rd\e+\OO\left(e^{-c(\log N)^2}\right),\label{eqn:polar2}
\end{align}
where in the last equality we used the simple yet useful facts that $\int_\C K_N(z,\bar z)\A(z)=1$ and
$|K_N(z,\bar z)-\frac{N}{\pi}\1_{|z|<1}|=\OO(e^{-c(\log N)^2})$ uniformly out of $\Omega_3$.
We now can write $g(1+\e)-g(1)=N^aG'(0)\e+\OO(N^{2a}\e^2)$, and use $(\ref{eqn:polar})$ to approximate
$K_N(1+\e,1+\e)-\frac{N}{2\pi}$.
Using the fact that $\mu(z)=\frac{|z-1|}{\sqrt{2}}+\OO((z-1)^2)$ for $z$ close to 1,
still denoting $t=\frac{N}{N-1}(1+\e)^2$, we have
$\mu(t)=\sqrt{2}\e+\OO(\e^2+N^{-1})$, so the integral term in $(\ref{eqn:polar2})$ can be written
\begin{multline*}
\frac{N}{\pi}\int_{-\e_N}^{\e_N}(1+\OO(\e))(N^aG'(0)\e+\OO(N^{2a}\e^2))\left(\frac{1}{2}+\OO(\e+N^{-1})\right)\\
\erfc(\sqrt{2N}|\e|+\OO(N^{-1/2}+N^{1/2}\e^2))\left(1+\OO(N^{-1/2})\right)
\rd\e
\end{multline*}
One can easily check by bounding with absolute values that all terms involving $\OO$'s , except the one in the $\erfc$ function, contribute to $\OO(N^{a-\frac{1}{2}+\e})$, so we just need to prove that
$$
N\int_{-\e_N}^{\e_N}\e
\erfc(\sqrt{2N}|\e|+\OO(N^{-1/2}+N^{1/2}\e^2))
\rd\e=\OO(N^{-\frac{1}{2}}).
$$
As $\erfc'=\OO(1)$ uniformly, this is the same as proving $N\int_{-\e_N}^{\e_N}\e
\erfc(\sqrt{2N}|\e|)
\rd\e=\OO(N^{-1/2+\e})$, which is obvious as it vanishes by parity.
\end{proof}

We now prove the convergence of the second moment.

\begin{lemma}\label{lem:2ndCum}
Under the hypothesis of Theorem \ref{z1Ginibre},
as $N\to\infty$,
$$
\cum_N(2)=
\frac{1}{4\pi}\int|\nabla f(z)|^2\mathds{1}_{z\in A}\rdA(z)
+\frac{1}{2}\|f_{\rm T}\|^2_{{\rm H}^{1/2}}+\oo(1)
$$
where we use the notation $A=\{z\mid\Re(z\overline{z_0})<0\}$, $f_{{\rm T}}(z)=f(\ii z_0 z)$ and $\|g\|^2_{{\rm H}^{1/2}}=\frac{1}{4\pi^2}\int_{\mathbb{R}^2}\left(\frac{g(x)-g(y)}{x-y}\right)^2\rd x\rd y$.
\end{lemma}

\begin{proof}
Note that by polarization of formula (\ref{eqn:Fl}),
\begin{align*}
\cum_N(2)&=\frac{1}{2}\int_{\C^2}(f^{(N)}(z_1)-f^{(N)}(z_2))^2|K_N(z_1,z_2)|^2\A(z_1,z_2)\\
&=\frac{1}{2}\int_{\Omega_1^2\cup \Omega_3^2}(f^{(N)}(z_1)-f^{(N)}(z_2))^2|K_N(z_1,z_2)|^2\A(z_1,z_2)+\OO\left(e^{-c(\log N)^2}\right),
\end{align*}
the contributions of the domains $\C\times\Omega_2$ and $\Omega_1\times\Omega_3$ being easily bounded by Lemma \ref{lem:fastDecrease}.
For the $\Omega_1^2$ term, one can easily reproduce the method employed in the bulk, in the previous subsection, approximating $K_N$ by $k_N$ thanks to Lemma \ref{lem:KnBulk}, to obtain that the contribution of the above integral on the domain $\Omega_1^2$ is
$
\frac{1}{4\pi}\|\chi_A\nabla f\|_2^2+\oo(1)$.

The most tricky part consists in evaluating the $\Omega_3^2$ term. To calculate it, we will need the following notations,
where $\theta_1=\arg(z_1)$, $\theta_2=\arg(z_2)$:
\begin{enumerate}[(i)]
\item $=\frac{1}{2}\int_{\Omega_3^2}(f^{(N)}(z_1)-f^{(N)}(z_2))^2|K_N(z_1,z_2)|^2\A(z_1,z_2),$
\item $=\frac{1}{2}\int_{\Omega}(f^{(N)}(z_1)-f^{(N)}(z_2))^2|K_N(z_1,z_2)|^2\A(z_1,z_2),$
where $\Omega=(\{(z_1,z_2)\in\Omega_3^2: |\theta_1-\theta_2|>N^{-\frac{1}{2}+\e}\})$, where $\e>0$ is fixed and will be chosen small enough,
\item $=\frac{1}{2}\int_{\Omega}(f^{(N)}(e^{\ii \theta_1})-f^{(N)}(e^{\ii \theta_2}))^2|K_N(z_1,z_2)|^2\A(z_1,z_2),$
\item  $=\frac{1}{2}\frac{N}{2\pi^3}\int_{\Omega}\frac{(f^{(N)}(e^{\ii \theta_1})-f^{(N)}(e^{\ii \theta_2}))^2}{\left|
e^{\ii\theta_1}-e^{\ii\theta_2}
\right|^2}e^{-N(|z_1|^2+|z_2|^2)}e^{2N}|z_1z_2|^{2N}\A(z_1,z_2),$
\item $=\frac{1}{8\pi^2}\int_{(-\pi,\pi)^2}\frac{(f^{(N)}(e^{\ii \theta_1})-f^{(N)}(e^{\ii \theta_2}))^2}{\left|
e^{\ii \theta_1}-e^{\ii \theta_2}
\right|^2}\rd \theta_1\rd \theta_2,$
\item $=\frac{1}{8\pi^2}\int_{\R^2}\frac{(f_{{\rm T}}(r_1)-f_{{\rm T}}(r_2))^2}{(r_1-r_2)^2}\rd r_1\rd r_2$.
\end{enumerate}
We will prove successively that the six above terms differ by $\oo(1)$ a $N\to\infty$. This will conclude the proof as
the last one is exactly $\frac{1}{2}\|f_{\rm T}\|^2_{{\rm H}^{1/2}}$.
First, the difference between (i) because (ii) is $\oo(1)$: it can be bounded by
$$
\e_N N^a\|f\|^2_\infty\sup_{z_1\in\Omega_3}\int_{|z_2-z_1|<N^{-\frac{1}{2}+\e}}|z_1-z_2|^2|K_N(z_1,z_2)|^2\A(z_2),
$$
where
$N^{-a}\e_N$ corresponds to the position of $z_1$ such that $|\theta_1-\theta_2|<N^{-\frac{1}{2}+\e}$ and $|f^{(N)}(z_1)-f^{(N)}(z_2)|\neq 0$,
and $N^{2a}\|f'\|^2_\infty$ comes from the Lipschitz property for $f^{(N)}$.
Moreover, lemmas \ref{lem:close1} and \ref{lem:far1} easily yield
\begin{equation}\label{eqn:sqrt}
|K_N(z_1,z_2)|=\OO\left(\frac{N}{1+\sqrt{N}|z_1-z_2|}\right),
\end{equation}
 uniformly in $\C^2$.
Hence the difference between (i) and (ii) is bounded by (we choose $\e<\frac{1}{2}\left(\frac{1}{2}-a\right)$)
$$
\e_N N^a\|f\|^2_\infty\int_{0<r<N^{-\frac{1}{2}+\e}}r^3\left(\frac{N}{1+\sqrt{N}r}\right)^2\rd r
=
\OO\left(\e_N N^aN^{2\e}\|f\|^2_\infty\right)=\oo(1).
$$

Moreover,  noting $A=\Omega
\cap (\{f^{(N)}(z_1)\neq 0\}\cup\{f^{(N)}(z_2)|\neq 0\})$ and doing anan order 1 approximation around
$e^{\ii \theta_1}$ and $e^{\ii \theta_2}$, (ii)-(iii) is of order at most
\begin{multline*}
\e_N\|{f^{(N)}}'\|_\infty\|f^{(N)}\|_\infty
\int_{A}|K_N(z_1,z_2)|^2\A(z_1,z_2)
\leq \e_N N^a \int_{A}\frac{N^2}{(1+\sqrt{N}|e^{\ii\theta_1}-e^{\ii\theta_2}|)^2}\A(z_1,z_2)\\
\leq \e_N^3N^2\int_{r>N^{-1/2+\e}}\frac{\rd r}{(1+\sqrt{N}r)^2}
\leq \e_N^3N^{-\frac{3}{2}-\e},
\end{multline*}
where the derivatives are radial ones, so (ii)$-$(iii)$=\oo(1)$.

For the difference between (iii) and (iv), we want to get a good approximation for $K_N(z_1,z_2)$. Note that,
as $|\theta_1-\theta_2|\gg N^{-1/2+\e}\gg2\e_N,$ one easily has that on $\Omega$, $|\arg(z_1\bar z_2-1)|\sim\frac{\pi}{2}$ uniformly, in particular the formula $(ii)$ in Lemma $\ref{lem:close1}$ provides a good approximation for
$e_{N-1}(Nz_1\bar z_2)$: together with the asymptotics $\erfc(z)\underset{|z|\to\infty}{\sim}\frac{e^{-z^2}}{z\sqrt{\pi}}$, for $|\arg z|<\frac{3\pi}{4}$, and noting that $\mu'(z)=\frac{z_1\bar z_2-1}{2\mu(z)z_1\bar z_2}$, we get
$$
e_{N-1}(Nz_1\bar z_2)\underset{N\to\infty}{\sim}\frac{1}{\sqrt{2\pi N}(z_1\bar z_2-1)}e^N(z_1\bar z_2)^N,
$$
so in particular
$$
|K_N(z_1,z_2)|^2\underset{N\to\infty}{\sim}\frac{N}{\pi^2}\frac{1}{2\pi|z_1\bar z_2-1|^2-2}e^{-N(|z_1|^2+|z_2|^2)}|z_1z_2|^{2N}
\underset{N\to\infty}{\sim}
\frac{N}{\pi^2}\frac{1}{2\pi|e^{\ii\theta_1}-e^{\ii\theta_2}|^2}e^{-N(|z_1|^2+|z_2|^2-2)}|z_1z_2|^{2N},
$$
the last equivalence relying still on the fact that, on $\Omega$, $|\theta_1-\theta_2|\gg |z_1-e^{\ii\theta_1}|,
|z_2-e^{\ii\theta_2}|$.
This proves that (iv)-(iii)=o((iii)).

To obtain (v)-(iv)=o((iv)), just note that $|z_1|^2-1-\log(|z_1|^2)=\frac{1}{2}(|z_1|^2-1)^2+\OO(||z_1|^2-1|^3)$,
so by a simple saddle point expansion we get
$$
\int_{1-\e_N}^{1+\e_N}xe^{-Nx^2+N}x^{2N}\rd x\underset{N\to\infty}{\sim}
\int_{1-\e_N}^{1+\e_N}xe^{-\frac{N}{2}(x^2-1)^2}\rd x
\underset{N\to\infty}{\sim}\frac{1}{2}
\int_{-\infty}^{\infty}e^{-\frac{N}{2}u^2}\rd u=\sqrt{\frac{\pi}{2N}}
$$
By integrating (iv) along the radial parts of $z_1,z_2$, we therefore get (v), up to the domain where $|\theta_1-\theta_2|<N^{-\frac{1}{2}+\e}$, which can be shown to have a negligible contribution, as previously.

Finally, (v) converges to (vi), by a simple  change of variables and dominated convergence.
\end{proof}

For all cumulants of order $\ell\geq 3$, we begin with the following very crude bound.

\begin{lemma}\label{lem:crude}
Under the conditions of Theorem D.4, for any fixed function $f$, $\e>0$, and
$
\ell\geq 3,
$
we have (as $N\to\infty$)
$$
\cum_N(\ell)=\OO(N^{3}).
$$
\end{lemma}

\begin{proof}
In the decomposition (\ref{cumAHM}), it is obvious that the terms $A_N(\ell), B_N(\ell), C_N(\ell), D_N(\ell)$
are  $\OO(N^3)$, just by bounding $K_N$ by $N$ in $\Omega_1\cup\Omega_3$ and
using lemma \ref{lem:fastDecrease} if one point of $K_N$ is in $\Omega_2$. Concerning the term
$E_N(\ell)$, by reproducing the argument in the proof of Theorem \ref{lclGinibre} we just need to prove that
it is $\OO(N^3)$ when all points in the integrand are restricted to $\Omega_3$, i.e. it would be sufficient to prove that
$$
\int_{\Omega_3^{\ell+1}}\left|K_N(z_0,z_1)K_N(z,z_1)\dots K_N(z_{\ell},z_0)\right|\A(z_0,\dots,z_\ell)=\OO(N^3).
$$
using the estimate (\ref{eqn:sqrt}) and integrating along the width of $w_3$, we just need to prove that for some $\e>0$,
\begin{equation}\label{eqn:intermed}
N^{\frac{\ell+1}{2}}\int_{(-1,1)^{\ell+1}}\frac{1}{1+\sqrt{N}|x_0-x_1|}\dots \frac{1}{1+\sqrt{N}|x_\ell-x_0|}\rd x_0\dots\rd x_\ell=\OO(N^{3-\e}).
\end{equation}
A simple calculation yields, for any $-1<a<b<1$,
$$
\int_{(-1,1)}\frac{1}{1+\sqrt{N}|a-u|}\frac{1}{1+\sqrt{N}|u-b|}\rd u=\OO\left(\frac{\log N}{\sqrt N}\right)\frac{1}{1+\sqrt{N}|a-b|},
$$
so by integrating successively the variables $x_2,\dots,x_\ell$ in (\ref{eqn:intermed}), we therefore obtain that a sufficient condition is
$
 N\int_{(-1,1)^2}\frac{1}{(1+\sqrt{N}|x_0-x_1|)^2}\rd x_0\rd x_1=\OO(N^{3-\e}),
$
which is obvious.
\end{proof}

Relying on the  Marcinkiewicz theorem, the above initial bounds on the cumulants can be widely improved.

\begin{lemma}
Under the conditions of Theorem D.4, for any fixed function $f$, $\e>0$, and
$
\ell\geq 3,
$
we have (as $N\to\infty$)
$$
\cum_N(\ell)=\OO(N^{\e}).
$$
\end{lemma}

\begin{proof}
Let $Y^{(N)}=N^{-\e}X^{(N)}_f$. Let $\tilde\cum_N(\ell)$ be the $\ell$-th cumulant of $Y^{(N)}$:
$\tilde\cum_N(\ell)=N^{-\ell\e}\cum_N(\ell)$. Hence, as a consequence of Lemma \ref{lem:crude},
there is a rank $\ell_0$ such that for any $\ell>\ell_0$,
$
\tilde\cum_N(\ell)\to 0
$
as $N\to\infty$.
We wish to prove that
\begin{equation}\label{eqn:boundedCum}
\sup_{\ell\in\llbracket 3,\ell_0\rrbracket}\tilde\cum_N(\ell)\underset{N\to\infty}{\longrightarrow}0
\end{equation}
as well. Assume the contrary. We therefore can find $\delta>0$ independent of $N$, an index $j_1\in\llbracket 3,\ell_0\rrbracket$ and a
subsequence $(N_k)$ such that
$$
|\tilde\cum_{N_k}(j_1)|>\delta,\
|\tilde\cum_{N_k}(j_1)|^{1/j_1}=\sup_{\ell\in\llbracket 3,\ell_0\rrbracket}\{|\tilde\cum_{N_k}(\ell)|^{1/\ell}\}.
$$
Consider now the random variable $Z^{(k)}=Y^{(N_k)}/|\tilde\cum_{N_k}(j_1)|^{1/j_1}$.
Then all the cumulants of $Z^{(k)}$ go to $0$ except the $j_1$-th one, equal to 1, and eventually some other cumulants with indexes in $\llbracket 3,\ell_0\rrbracket$, which are all uniformly bounded. Let $\mu_k$ be the distribution of $Z^{(k)}$.
As the first and second cumulants of $=Z^{(k)}$ are uniformly bounded,
 $\E({Z^{(k)}}^2)$ is uniformly bounded so the sequence
$(\mu_k)_{k\geq 0}$ is tight.

We therefore can choose a subsequence of $(\mu_k)$, called $(\mu_{i_k})$, which converges weakly to
a measure $\nu$, and we can moreover assume that all of its cumulants (which are uniformly bounded) with indexes in
$\llbracket 3,\ell_0\rrbracket$ converge. Let $Z$ be a random variable with distribution $\nu$.
As $\mu_{i_k}$ has any given cumulant uniformly bounded, it has any given moment uniformly
bounded, so by the Corollary of Theorem 25.12 in \cite{Bil1995}, $Z$ has moments of all orders, and those of
$Z^{(i_k)}$ converge to those of $Z$. Hence $Z$ has cumulants of all orders and those of $Z^{(i_k)}$
converge to those of $Z$. Hence $\nu$ has all of its cumulants equal to 0 for $\ell>\ell_0$, and its cumulant of order
$j_1\in\llbracket 3,\ell_0\rrbracket$ equal to 1. From the  Marcinkiewicz theorem (cf. the Corollary to
Theorem 7.3.2 in \cite{Luk1960}),
such a distribution $\nu$ does not exist, a contradiction.

Equation (\ref{eqn:boundedCum}) therefore holds, so we proved that for any arbitrary $\e>0$ and any $\ell\geq 1$,
$\cum_{N}(\ell)=\OO(N^{\ell\e})$. For fixed $\ell$, by choosing $\e$ small enough, we get the result.
\end{proof}

Finally, we note that the previous bounds on all cumulants of $X_f^{(N)}$, each of order at most $N^\e$, allow to conclude the proof of Theorem
\ref{z1Ginibre}, by the Markov inequality.
The explicit asymptotic form of the second cumulant (Lemma
\ref{lem:2ndCum}) is more than what we need, we made the calculus explicit as one expects that $X^{(N)}_f$
converges to a Gaussian centered random variable with asymptotic
variance $\frac{1}{4\pi}\int|\nabla f(z)|^2\mathds{1}_{z\in A}\rdA(z)
+\frac{1}{2}\|f_{\rm T}\|^2_{{\rm H}^{1/2}}$.

\subsection{Estimates on the partial exponential function. }

For the following lemma we need the function $\erfc(z)=\frac{2}{\sqrt{\pi}}\int_z^{+\infty}e^{-\omega^2}\rd\omega$.

\begin{lemma}[Bleher, Mallison \cite{BleMal2006}]\label{lem:close1}
For $\delta>0$ small enough, let $\mu(z)=\sqrt{z-\log z-1}$, be uniquely defined as analytic in
$D(1,\delta)$ and $\mu(1+x)>0$ for $0<x<\delta$. Then for any $M>1$, as $N\to\infty$,
$e^{-N z}e_N(Nz)$ has the following asymptotics:
\begin{enumerate}[(i)]
\item $\frac{1}{2}\erfc(\sqrt{N}\mu(z))+\OO(N^{-1/2})$ if $|z-1|<\frac{M}{\sqrt{N}}$;
\item $\frac{1}{2\sqrt{2}\mu'(z)}\erfc(\sqrt{N}\mu(z))\left(1+\OO\left(\frac{1}{(z-1)N}\right)\right)$ if
$\frac{M}{\sqrt{N}}\leq|z-1|\leq\delta$ and $|\arg(z-1)|\leq \frac{2\pi}{3}$;
\item $1-\frac{1}{2\sqrt{2}\mu'(z)}\erfc(-\sqrt{N}\mu(z))\left(1+\OO\left(\frac{1}{(z-1)N}\right)\right)$ if
$\frac{M}{\sqrt{N}}\leq|z-1|\leq\delta$ and $|\arg(z-1)-\pi|\leq \frac{2\pi}{3}$.
\end{enumerate}
\end{lemma}

In case of real $z$, the above estimates can be gathered as follows.

\begin{lemma}[Wimp \cite{BoyGoh2007}]\label{lem:close1Real}
Uniformly for $t\geq 0$,
$$
e^{-N t}e_N(Nt)=\1_{0\leq t< 1}+\frac{1}{\sqrt{2}}\frac{\mu(t)t}{t-1}\erfc(\sqrt{N}\mu(t))
\left(1+\OO\left(\frac{1}{\sqrt{N}}\right)\right),
$$
where $\mu(t)=\sqrt{t-\log t-1}$ is defined to be positive for all $t$, contrary to Lemma \ref{lem:close1}.
\end{lemma}

Note that in the above lemma, the asymptotics are coherent at $t=1$ because
$$
\frac{1}{\sqrt{2}}\frac{\mu(t)t}{t-1}\longrightarrow
\left\{
\begin{array}{rl}
-\frac{1}{2}&{\rm as}\ t\to 1^-\\
\frac{1}{2}&{\rm as}\ t\to 1^+\\
\end{array}
\right..
$$

\begin{lemma}[Kriecherbauer, Kuijlaars, McLaughlin, Miller \cite{KriKuiMcLMil2008}]\label{lem:far1}
For $0<a<1/2$, let $U=\{|z-1|<N^{-a}\}$. Then there exists polynomials $h_j$ of degree $2j$ such that for all
$r\in\N^*$ we have
$$
e_{N-1}(Nz)=
\left\{\begin{array}{ll}
e^{Nz}\left(1-e^Nz^Ne^{-Nz}F_N(z)\right)& \mbox{for}\ z\in D-U,\\
e^{N}z^N F_N(z)& \mbox{for}\ z\in D^c-U,
\end{array}
\right.
$$
where
$$
F_N(z)=\frac{1}{\sqrt{2\pi N}(1-z)}\left(1+\sum_{j=1}^{r-1}\frac{h_j(z)}{N^j(z-1)^{2j}}+\OO\left(\frac{1}{N^r|1-z|^{2r}}\right)\right)
$$
uniformly in $\C-U$.
\end{lemma}

\end{document}